\documentclass[a4paper]{amsart}

\usepackage{tikz}
\usepackage{pgflibraryarrows}
\usepackage{pgflibrarysnakes}

\usepackage{amsfonts}
\usepackage{amssymb}
\usepackage{amsthm}
\usepackage{amsmath}
\usepackage[all]{xy}

\usepackage{pb-diagram}

\SelectTips{cm}{}

\theoremstyle{plain}
\newtheorem{thm}{Theorem}[section]
\newtheorem{prop}[thm]{Proposition}
\newtheorem{cor}[thm]{Corollary}
\newtheorem{lem}[thm]{Lemma}

\theoremstyle{definition}
\newtheorem{defn}[thm]{Definition}

\theoremstyle{remark}
\newtheorem{rem}[thm]{Remark}
\newtheorem{expl}[thm]{Example}

\numberwithin{equation}{section}

\newcommand{\sC}{\mathcal{C}}

\newcommand{\sI}{\mathcal{I}}

\newcommand{\sS}{\mathcal{S}}

\newcommand{\bL}{\mathbf{L}}

\newcommand{\wrho}{\widetilde{\rho}}

\newcommand{\cyl}{\textup{cyl}}
\newcommand{\hofib}{\textup{hofib}}

\renewcommand{\AA}{\mathbb{A}}
\newcommand{\BB}{\mathbb{B}}
\newcommand{\CC}{\mathbb{C}}
\newcommand{\DD}{\mathbb{D}}
\newcommand{\FF}{\mathbb{F}}
\newcommand{\HH}{\mathbb{H}}
\newcommand{\NN}{\mathbb{N}}

\newcommand{\QQ}{\mathbb{Q}}
\renewcommand{\SS}{\mathbb{S}}
\newcommand{\ZZ}{\mathbb{Z}}

\newcommand{\id}{\textup{id}}

\newcommand{\Hom}{\textup{Hom}}
\newcommand{\Mor}{\textup{Mor}}

\newcommand{\nin}{\noindent}

\newcommand{\ra}{\rightarrow}
\newcommand{\xra}{\xrightarrow}
\newcommand{\lra}{\longrightarrow}
\newcommand{\co}{\colon\!}

\newcommand{\G}{\textup{G}}
\newcommand{\TOP}{\textup{TOP}}
\newcommand{\STOP}{\textup{STOP}}

\newcommand{\Gsign}{\textup{G-sign}}
\newcommand{\sign}{\textup{sign}}

\renewcommand{\int}{\textup{int}}

\newcommand{\RhG}{R_{\widehat G}}

\newcommand{\pr}{\textup{pr}}

\newcommand\C{\mathbb{C}}
\newcommand\h{\mathbb{H}}

\newcommand\R{\mathbb{R}}

\newcommand\Q{\mathbb{Q}}
\newcommand\Z{\mathbb{Z}}
\newcommand\F{\mathbb{F}}

\newcommand\tensor{\otimes}

\newcommand\wt[1]{\widetilde{#1}}

\newcommand\del{\partial}

\newtheorem{Theorem}{Theorem}[section]

\newtheorem{Example}[Theorem]{Example}

\newtheorem{Remark}[Theorem]{Remark}

\title[The $\rho$-invariant and periodicity in topological surgery]{The additivity of the $\rho$-invariant and periodicity \\ in topological surgery}

\author{Diarmuid Crowley, Tibor Macko}

\date{\today}



\address{Hausdorff Research Institute for Mathematics\\
Poppelsdorfer Allee 82\\
D-53115 Bonn \\
Germany} \email{diarmuidc23@gmail.com}

\address{Mathematisches Institut \\ Universit\"at M\"unster \\
Einsteinstra{\ss}e 62 \\ M\"unster, D-48149 \\ Germany \\ and
Matematick\'y \'Ustav SAV \\ \v Stef\'anikova 49 \\ Bratislava,
SK-81473 \\ Slovakia} \email{macko@uni-muenster.de}

\thanks{The second author was supported by SFB 478 Geometrische
Strukturen in der Mathematik, M\"unster.}

\begin{document}

\maketitle

\begin{abstract}
For a closed topological manifold $M$ with $\dim (M) \geq 5$ the
topological structure set $\sS(M)$ admits an abelian group structure
which may be identified with the algebraic structure group of $M$
as defined by Ranicki. If $\dim (M) = 2d-1$,  $M$ is
oriented and $M$ is equipped with a map to the classifying space of a
finite group $G$, then the reduced $\rho$-invariant defines a
function,
\[
\wt \rho : \sS(M) \to \QQ \RhG^{(-1)^d},
\]
to a certain sub-quotient of the complex representation ring of $G$.
We show that the function $\wrho$ is a homomorphism when $2d-1 \geq
5$.

Along the way we give a detailed proof that a geometrically defined
map due to Cappell and Weinberger realises the 8-fold Siebenmann
periodicity map in topological surgery.
\end{abstract}


\section{Introduction} \label{sec:intro}

Let $M$ be a closed oriented $(2d-1)$-dimensional topological
manifold and let $\lambda (M) \co M \ra BG$ be a map to the
classifying space of a finite group $G$.  The  $\rho$-invariant of
$(M, \lambda (M))$,
\[ \rho (M,\lambda (M)) \in \QQ \RhG^{(-1)^d}, \]
lies in a certain sub-quotient of the rationalised complex
representation ring of $G$ (see Section \ref{subsec:rho-inv} for
details).  It is a powerful invariant of odd-dimensional manifolds
with torsion elements in their fundamental group.  To mention just
two examples: it was used by Atiyah and Bott to show that two smooth
lens spaces which are $h$-cobordant are diffeomorphic
\cite{Atiyah-Bott(1967)}. It also plays a key role in Wall's
classification results for fake lens spaces in the piecewise linear and topological categories  \cite[Chapter
14]{Wall(1999)}.


Assume now that $2d-1\geq 5$ and consider $\sS (M)$, the topological
structure set of $M$.  The elements of $\sS (M)$ are homotopy
equivalences $h \co N \ra M$ of closed manifolds modulo the
$h$-cobordism relation in the source.\footnote{Our results work
equally well for the simple structure set, see Remark
\ref{rem:decorations}.}  We define the reduced $\rho$-invariant by
\begin{equation} \label{formula:reduced-rho}
\wrho \co \sS (M) \lra  \QQ \RhG^{(-1)^d}, ~~~~~[h: N \ra M] ~
\longmapsto ~ \rho (N, \lambda  \circ h) - \rho(M,\lambda).
\end{equation}
A feature of topological surgery is that $\sS(M)$ admits the
structure of an abelian group which is natural in some sense
\cite{Siebenmann(1977), Ranicki(1978)}. Since this group structure
on $\sS(M)$ is mysterious from the geometric point of view it is
not clear whether $\wrho$ is a homomorphism of abelian groups.

It is clear, however, that $\wt \rho$ is additive with respect to
the action of the $L$-group on $\sS(M)$.  Let $\pi = \pi_1(M)$ and
recall that the surgery group $L_{2d}(\pi)$ acts on $\sS(M)$ via
Wall-realisation and also that the induced homomorphism $\lambda
(M)_* : \pi \to G$ together with the $G$-signature define a
homomorphism $\sigma_{\lambda (M)} \co L_{2d}(\pi) \to \QQ
\RhG^{(-1)^d} $.  It is well known that this action is additive with
respect to $\wrho$ (\cite{Petrie(1970)}):  if $x \in L_{2d}(\pi)$
and $[h] \in \sS(M)$ then
\begin{equation} \label{eqn:additivity-w-r-t-action-of-L}
\wrho([h] + x) = \wrho([h]) + \sigma_\lambda(x).
\end{equation}

Moreover, calculations in \cite[Chapter 14E]{Wall(1999)} and
\cite{Macko-Wegner(2008)} show that $\wrho$ is a homomorphism when
$M$ is a lens space.  Wolfgang L{\"u}ck asked whether this is true in
general and a positive answer to this question is our main theorem.


\begin{thm}  \label{thm:main}
Let $M$ be a closed oriented topological manifold of dimension $2d-1 \geq 5$ with a reference map $\lambda (M) \co M \ra BG$ where $G$ is a finite group. Then the map
\[ \wrho \co \sS (M) \lra \QQ \RhG^{(-1)^d} \]
is a homomorphism of abelian groups.
\end{thm}

We see that Theorem \ref{thm:main} is a generalisation of the long
standing identity (\ref{eqn:additivity-w-r-t-action-of-L}). One may
also take the point of view that it sheds light on the group
structure on $\sS(M)$.   Clearly it has the potential to aid in
computations of $\sS(M)$ and this is shown to be the case in a
forthcoming paper of Davis and L{\"u}ck \cite{Davis-Lueck(2010)}
about torus bundles over lens spaces.  Clearly we also have

\begin{cor}
The map $\wrho$ factors through $\sS(M) {\longrightarrow} \sS(M) \otimes \QQ$.
\end{cor}



\subsection{The outline of the proof of Theorem \ref{thm:main}} \label{sec:outline}


To describe the essential ideas of the proof, we first sketch the
topological definition of the $\rho$-invariant which we use
throughout the paper.  Let $(M,  \lambda (M))$ be as above.  If $Z$
is a compact oriented $2d$-dimensional manifold with a map
$\lambda(Z) \co Z \to BG$, we call it an $r$-coboundary for $(M,
\lambda (M))$ if $\del (Z, \lambda(Z)) = \sqcup_r (M, \lambda (M))$
is the disjoint union of $r$ copies of $(M, \lambda (M))$ for some
$r \geq 1$.  From bordism theory we know that $r$-coboundaries
always exist for some $r$.  The $G$-signature of the induced
$G$-covering $\widetilde Z$ is an element in the complex
representation ring $R(G)$. It follows from the Atiyah-Singer
$G$-index theorem \cite{Atiyah-Singer-III(1968)}, \cite[Chapter
14B]{Wall(1999)} that the expression
\[
\rho(M,\lambda(M)) : = (1/r) \cdot \Gsign (\widetilde Z)
\]
becomes independent of the choice of $Z$ and $r \geq 1$ after
passing to a certain subquotient of the rationalisation of $R(G)$
(see Definition \ref{defn:rho-1} for a precise statement).

Suppose now that we have structures $h_0 \co N_0 \to M$ and $h_1 \co
N_1 \to M$ representing two elements in $\sS (M)$.  Unless $[h_1] =
[{\rm id}] + x$ for some $x \in L_{2d}(\pi)$, a geometric description of
the structure $[h_0] + [h_1]$ in terms of $[h_0]$ and $[h_1]$ is not
known at present.  Thus it is not a-priori clear how to relate
$r$-coboundaries for $N_0$ and $N_1$ and one sees that the
additivity of the function $\wt \rho$ from Theorem
\ref{formula:reduced-rho} is not obvious.

On the other hand, the situation becomes much simpler if we replace
the closed manifold $M$ by $M \times D^l$ for some $l \geq 1$ as we
now describe.  The rel boundary structure set of $M \times D^l$,
$\sS_\del(M \times D^{l})$, consists of equivalence classes of
homotopy equivalences of manifolds with boundary $h \co (N,\del) \to
(M \times D^l,\del)$, such that the restriction to the boundary is a
homeomorphism $\del h \co \del N \cong M \times S^{l-1}$.  The
equivalence relation is given by $h$-cobordism of pairs in the
source where the h-cobordism over $\del N$ is trivial. For $l \geq 1$ there is a geometrically defined group structure using ``stacking'' which is easy to understand: see
Definition \ref{defn:stacking}.  Suppose that $n+l = \dim (M) + l =
2d-1$. Then following \cite{Madsen-Rothenberg-II(1989)}, we define
the rel boundary reduced $\rho$-invariant
\begin{equation} \label{formula:relative--reduced-rho}
\wrho_\del : \mathcal{S}_\del^{}(M \times D^l) \lra \QQ
\RhG^{(-1)^d}, ~~~~~[h \co N \ra M \times D^l] ~ \longmapsto ~
\rho(N \cup_{\del h} (M \times D^l)).
\end{equation}
The reference maps are left out of the notation. Notice that $N
\cup_{\del h} (M \times D^l)$ is a closed oriented
$(2d-1)$-dimensional manifold and so the formula makes sense.  Using
a certain generalised connected sum operation we prove
\begin{prop} \label{prop-C}
Let $M$ be a closed oriented topological manifold of dimension $n$
with a reference map $\lambda (M) \co M \ra BG$ for a finite group
$G$, and let $n+l = 2d-1 \geq 5$. Then the map
\[
\wrho_\del \co \sS_\partial (M \times D^l) \lra \QQ \RhG^{(-1)^d}
\]
is a homomorphism of abelian groups.
\end{prop}

For reasons that will become apparent later we choose $l = 4j$, and
contemplate the following diagram.
\[
\xymatrix{
\sS (M) \ar[dr]_{\wrho} & ??? & \sS_\partial (M \times D^{4j}) \ar[dl]^{\wrho_\partial} \\
& \QQ \RhG^{(-1)^d} }
\]
If we can find a homomorphism $\sS(M) \to \sS_\del(M \times D^{4j})$
making the above diagram commute then $\wt \rho$ is a homomorphism
and we are done.  This brings us to periodicity in topological
surgery which we discuss in more detail in Section
\ref{subsec:periodhist} below.  For now we simply note that there is
an injective near periodicity map $P^j : \sS(M) \to \mathcal{S}_\del(M
\times D^{4j})$ defined in \cite{Siebenmann(1977)} and in a
different way in \cite{Ranicki(1978)} and \cite{Ranicki(1992)}.  However both definitions are complicated and require one to travel a long journey away from the geometry of a
structure $[h : N \to M] \in \sS (M)$. The distance is large enough
that we lose sight of $r$-coboundaries and so of the
$\rho$-invariant.

A geometric passage from $\sS (M)$ to $\sS_\del (M \times D^{4j})$
remained unclear until \cite{Cappell-Weinberger(1985)} where Cappell
and Weinberger sketched maps $CW^j \co \sS (M) \to \sS_\del (M
\times D^{4j})$ for $j = 1, 2$ or $4$.
However, their construction was given using piecewise linear
techniques and so strictly applies only when all manifolds involved
are triangulable, although the authors hinted at the
generalisations needed for the topological case. They claimed that
$CW^j = P^j$ but their proof uses Sullivan's Characteristic Variety
Theorem which was never published in sufficient generality. Later,
Hutt tried to address these issues \cite{Hutt(1998)}. He gave a
construction of the map $CW^1$ for topological manifolds.  However Hutt's
proof of near $4$-periodicity uses his own theory of
Poincar{\'e} sheaves which was never published.


Much of the work in this paper goes into giving a detailed proof
that the Hutt construction adapted to the map $CW^2$ indeed realises
the near periodicity map $P^2$.  In particular we replace Hutt's use of
Poincar\'{e} sheaves with algebraic surgery from
\cite{Ranicki(1992)} and thereby prove
%
\begin{thm} \label{thm:A}
Let $M$ be a closed topological manifold of dimension $n \geq 5$.
The Hutt description of the Cappell-Weinberger map gives an exact
sequence of homomorphisms of abelian groups:
\[
0 \lra \sS (M) \stackrel{CW^2}{\xrightarrow{\hspace*{0.75cm}}} \sS_\partial (M \times D^{8})
\lra H_0(M; \Z).
\]
\end{thm}
The details of the Hutt construction of the map $CW^j$ allow us to
do the following:  given $(Z, \lambda(Z))$, an $r$-coboundary for
$N$, the domain of a structure $[h: N \to M] \in \sS(M)$, we can
construct an $r$-coboundary for the domain of $CW^j([h])$.  This
then allows us to prove
\begin{thm} \label{thm:B}
Let $M$ be a closed topological manifold of dimension $(2d-1) \geq
5$ with a reference map $\lambda \co M \ra BG$ for a finite group
$G$. Then the following diagram commutes.
\[
\xymatrix{
\sS (M) \ar[rr]^{CW^2} \ar[dr]_{\wrho} & & \sS_\partial (M \times D^8) \ar[dl]^{\wrho_\partial} \\
& \QQ \RhG^{(-1)^d} }
\]
\end{thm}
Theorem \ref{thm:main} now follows immediately from Proposition
\ref{prop-C} and Theorems \ref{thm:A} and \ref{thm:B} since together
they show that $\wrho = \wrho_\del \circ CW^2$ is a composition of
homomorphisms.

\begin{rem}
The idea of understanding the group structure on $\sS(M)$ via the
stacking group structure on $\sS_\del(M \times D^{4j})$ and periodicity
is very natural.  For example in \cite{Jahren-Kwasik(2008)} Jahren
and Kwasik used this method to solve an extension problem for
$\sS(S^1 \times \R P^n)$ related to the Browder-Livesay invariant, a
close cousin of the $\rho$-invariant.
\end{rem}

\subsection{Periodicity in topological surgery} \label{subsec:periodhist}


In this subsection we briefly recall the history of periodicity in
topological surgery as well as describing how this paper adds to the
detailed proof of near periodicity.
Let $M$ be a closed topological manifold of dimension $n \geq 5$.
The source of periodicity in topological surgery is the 4-fold periodicity of the
homotopy groups $\pi_i(G/\TOP) \cong \pi_{i+4}(G/\TOP)$ for $i \geq
1$.  However, it took Quinn's theory of surgery spaces
\cite{Quinn(1970)} to see how this periodicity could be extended to
the structure set.  Once the surgery exact sequence was identified
as the long exact homotopy sequence of a fibration, Siebenmann
\cite{Siebenmann(1977)} could define injective maps $P^j : \sS(M)
\to \sS_\del (M \times D^{4j}) $\footnote{In fact Siebenmann
mistakenly claimed that $P^j$ is a bijection.   In general ${\rm
Im}(P^j)$ is a subgroup with $\sS(M \times D^{4j})/{\rm Im}(P^j)$
isomorphic to a subgroup of $H_0(M; \Z)$.  Therefore to be precise we speak of near-periodicity.  A correct statement of near periodicity appeared in \cite{Nicas(1982)}.}.  He used these maps to define an
abelian group structure on $\sS(M)$.

In \cite{Ranicki(1978)} and \cite{Ranicki(1992)} Ranicki produced
algebraic versions of surgery theory which translate Quinn's theory
into a category of chain complexes.  In particular bijections
\[ s : \sS_\del (M \times D^l) \to \SS_{n+l+1}(M)\]
are defined where $\SS_{n+l+1}(M)$ is an abelian group.  Moreover,
with respect to the stacking group structure on $\sS_\del(M \times
D^l)$ this map is an isomorphism if $l \geq 1$.  Since the algebraic
groups are nearly $4$-periodic almost by definition, Ranicki was able
to give an algebraic proof of Siebenmann's periodicity theorem. In
particular the algebraic theory of surgery so closely mirrors
surgery spaces that the two group structures defined on $\sS(M)$ agree.

As we have seen, for certain purposes the abstract descriptions of
the maps $P^j$ do not suffice and the papers of
\cite{Cappell-Weinberger(1985)} and \cite{Hutt(1998)} were written to
fill this gap.  For reasons mentioned above, however, neither of
these papers gives a water tight proof that the maps $CW^j : \sS(M)
\to \sS_\del(M \times D^{4j})$ realise the periodicity maps $P^j$.
In the end, to give a detailed proof that $CW^2 = P^2$ we have had
to combine important ideas from both papers and add some of our own.

For the outline of the proof of periodicity we were able to follow
\cite{Hutt(1998)}.  However to Hutt's arguments one must add
foundational results of \cite{Hughes-Taylor-Williams(1990)} and a
folk theorem proved in \cite{Hughes(1999)} about mapping cylinder
neighbourhoods, MCNs, and manifold approximate fibrations, MAFs.  We
summarise these results in Theorem \ref{thm:mafs-vs-mcns} and
Corollary \ref{cor:relative-mafs-vs-mcns} and use them to show that
Hutt's map is defined.  Then one has to take more care than Hutt to
show that the map is well-defined.  To show that the now
well-defined map $CW^j$ is indeed $P^j$ we use algebraic surgery
which requires an inductive dissection of a topological manifold
similar to, but not in general the same as, a simplicial
decomposition.  In particular algebraic surgery requires that we
apply Hutt's construction inductively to each space in such a
dissection.   Then one discovers that Theorem \ref{thm:mafs-vs-mcns}
concerning MCNs and MAFs has dimension restrictions which can only
be satisfied for 8-periodicity.  Thus we show that $CW^2 = P^2$ and
this is sufficient to prove the additivity of the reduced
$\rho$-invariant.  We hope that the work in this paper might serve
as a foundation to at last give a detailed proof that $CW^1 = P^1$.

\begin{rem} \label{rem:decorations}
All the results of the present paper work equally well for structure
sets and simple structure sets. Thus in the familiar notation, the
reader may substitute $\sS^s(M)$ or $\sS^h(M)$ for $\sS(M)$ and its
variants throughout the paper. To justify this we note that the
$\rho$-invariant is an $h$-cobordism invariant and hence defines a
function of both versions of the structure set.  Moreover the
forgetful map
$$ \sS^s(M) \to \sS^h(M)$$
is a homomorphism.  In particular the results of \cite{Quinn(1970)},
\cite{Siebenmann(1977)}, \cite{Ranicki(1978)} and
\cite{Ranicki(1992)} work equally well for each torsion decoration.
With regard to the periodicity maps $CW^j$ our treatment is also
simultaneous for both decorations: we use $h$-cobordisms throughout, but
the arguments are verbally the same with $s$-cobordisms.  In this
direction our work generalises \cite{Cappell-Weinberger(1985)} and \cite{Hutt(1998)} who only deal
with the $s$-decoration.
\end{rem}

The rest of the paper is organised as follows. In Section
\ref{sec:rho-inv} we define the $\rho$-invariant and its reduced
variations.  We also recall the group structure on $\sS_\del (M \times
D^l)$ for $l \geq 1$ and we prove Proposition \ref{prop-C}.  The
proof of Theorem \ref{thm:A} occupies Sections
\ref{sec:mafs-and-mcns}-\ref{sec:sieb-per}. In the preparatory
Section \ref{sec:mafs-and-mcns} we recall and reformulate essential
facts about MCNs and MAFs.  In Section \ref{sec:cw-map} we review
Hutt's account of the construction of the Cappell-Weinberger map. In
Section \ref{sec:alg-sur} we review the framework of the algebraic
theory of surgery from \cite{Ranicki(1992)} which is the key tool in
the proof of Theorem \ref{thm:A} in Section \ref{sec:sieb-per}. The
proof of Theorem \ref{thm:B} occupies the last two sections. Section
\ref{sec:bordism-groups} is again preparatory and the proof is
completed in Section \ref{sec:completion}.

{\bf Acknowledgements.} We would like to thank Jim Davis and Andrew
Ranicki for helpful discussions, Bruce Hughes for enlightening
correspondence at a critical juncture and Wolfgang L{\"u}ck for
raising the motivating question of the paper as well as helpful
discussions.  We would also like to thank the referee for helping to clarify a number of our proofs.

\section{The $\rho$-invariant} \label{sec:rho-inv}

Let $M$ be a closed oriented topological manifold of dimension $n =
2d-1 \geq 5$ with a reference map $\lambda (M) \co M \ra BG$ where
$G$ is a finite group. In this section we recall the definition of
the reduced $\rho$-invariant function, denoted $\wrho$, defined on the
structure set $\sS (M)$ as well as a relative $\rho$-invariant,
denoted $\wrho_\del$, which is defined on the rel boundary
structure set $ \sS_\del(M \times D^{2j})$.  The main outcome of the
section is the proof of Proposition \ref{prop-C} which states that
$\wrho_\del$ is a homomorphism.

\subsection{The $\rho$-invariant.}  \label{subsec:rho-inv}
The $\rho$-invariant is an invariant of odd-dimensional manifolds
associated to the $G$-signature of cobounding even-dimensional
manifolds. We first briefly recall the $G$-signature.

\nin \textbf{G-signature.} Let $G$ be a finite group acting smoothly
on a smooth manifold $Z^{2d}$. The rational intersection form of $Z$
is then a non-degenerate $(-1)^d$-symmetric bilinear form on which
$G$ acts. One can complexify the form and consider the positive and
negative definite $\CC$-vector subspaces. These are $G$-invariant
and hence define $G$-representations which can be subtracted in the
representation ring $R_\CC (G)$. The virtual representation  thus
obtained is denoted by $\Gsign (Z)$.  Complex conjugation induces an
involution on $R_{\CC} (G)$ with $(\pm 1)$-eigenspaces. In terms of
characters the $(+1)$-eigenspace corresponds to real characters and the
$(-1)$-eigenspace corresponds to purely imaginary characters. We
will denote
\[
R_{\CC}^{\pm} (G) : = \{ \chi \pm \chi^{-1} \; | \; \chi \in R_{\CC} (G) \}.
\]
One can also show that $\Gsign (Z) \in R^{(-1)^d} (G)$ which in
terms of characters means that we obtain a real (purely imaginary)
character, which will be denoted as $\Gsign (-,Z) \co g \in G
\mapsto \Gsign (g,Z) \in \CC$. The cohomological version of the
Atiyah-Singer $G$-index theorem, \cite[Theorem 6.12]{Atiyah-Singer-III(1968)}, tells us that if $Z$ is closed then
for all $g \in G$
\begin{equation} \label{ASGIT}
\Gsign (g,Z)  = L(g,Z) \in \CC,
\end{equation}
where $L(g,Z)$ is an expression obtained by evaluating certain
cohomological classes on the fundamental classes of the $g$-fixed
point submanifolds $Z^g$ of $Z$. In particular if the action is free
then $\Gsign (g,Z) = 0$ if $g \neq 1$. This means that $\Gsign (Z)$
is a multiple of the regular representation. This theorem was
generalised by Wall to topological semifree actions on topological
manifolds, which is the case we will need in this paper
\cite[Chapter 14B]{Wall(1999)}. The assumption that $Z$ is closed is
essential here, and motivates the definition of the
$\rho$-invariant.

\nin \textbf{Bordism groups.} To define the $\rho$-invariant one
also needs the following result which starts with the work of Conner
and Floyd \cite{Conner-Floyd(1964)} on smooth bordism, proceeds
through \cite{Williamson(1966)} for piecewise linear bordism and
finishes with \cite{Madsen-Milgram(1979)} for topological bordism.
\begin{thm} \label{thm:bord}
Let $G$ be a finite group with classifying space $BG$ and let
$\Omega_{n}^{\STOP} (BG)$ denote bordism group of $n$-dimensional
closed oriented topological manifolds with a reference map to
$BG$. Then for $2d-1 \geq 1$,
\[
\Omega_{2d-1}^{\STOP} (BG) \otimes \QQ = 0.
\]
\end{thm}
Let $N$ be a closed $(2d-1)$-dimensional manifold with a reference
map $\lambda (N) \co N \ra BG$ inducing a homomorphism $\lambda
(N)_\ast \co \pi_1 (N) \ra G$. The above result means that there
exists a $2d$-dimensional manifold with boundary $Z$ with a
reference map $\lambda (Z) \co Z \ra BG$ inducing a homomorphism
$\lambda (Z)_\ast \co \pi_1 (Z) \ra G$ such that $\del Z = r \cdot
N$ for some $r \geq 1$ and such that the restriction $\lambda
(Z)|_{\del Z} = r \cdot \lambda (N)$. Then we also have the induced
$G$-covering $\widetilde Z$ on which the group $G$ acts freely via
deck transformations. It is a manifold with boundary $r \cdot
\widetilde Y$, $r$ copies of the induced $G$-covering of $Y$.

The above considerations make it possible to make the following
definition.

\begin{defn}{\cite[Section 7]{Atiyah-Singer-III(1968)}} \label{defn:rho-1}
Let $N$ be a closed topological $(2d-1)$-dimensional manifold with a
reference map $\lambda (N) \co N \ra BG$ where $G$ is a finite group.
Define
\begin{equation}
\rho (N,\lambda (N)) : = \frac{1}{r} \cdot \Gsign(\widetilde Z) \in
\QQ R^{(-1)^d} (G)/ \langle \textup{reg} \rangle =: \QQ
\RhG^{(-1)^d}
\end{equation}
for some $r \in \NN$ and $(Z,\partial Z)$ such that $\partial Z = r
\cdot N$ and there is $\lambda (Z) \co Z \ra BG$ restricting to $r
\cdot \lambda (N)$ on $\del Z$. The symbol $\langle \textup{reg}
\rangle$ denotes the ideal generated by the regular representation,
the symbol $\QQ R^\pm (G)$ means $\QQ \otimes R^\pm (G)$.
\end{defn}

\noindent The $\rho$-invariant is well defined by the Atiyah-Singer
$G$-index theorem \cite[Theorem (6.12)]{Atiyah-Singer-III(1968)} and
its topological generalisation \cite[Chapter 14B]{Wall(1999)}. When
the reference map is clear we will often leave out the map $\lambda
(N)$ from the notation and simply write $\rho (N)$.

\subsection{Structure sets}

The structure set of a compact topological manifold is the basic
object of study in surgery theory. When calculated it gives us
understanding of the manifolds in the homotopy type of that given
manifold, as shown for example in \cite[Part 3]{Wall(1999)}.

\begin{defn} \label{defn:structure-set}
Let $M$ be a compact $n$-dimensional manifold with boundary
$\partial M$ (which may be empty).  A (simple) manifold structure on
$M$ relative to $\del M$ consists of a (simple) homotopy equivalence
of $n$-dimensional compact manifold with boundary
\[
(h,\partial h) \co (N,\partial N) \ra (M,\partial M)
\]
such that $\partial h$ is a homeomorphism. Two such structures
$(h_1,\partial h_1)$ and $(h_2, \partial h_2)$ are equivalent if
there exists a (simple) homotopy equivalence of $(n+1)$-dimensional
manifold $4$-ads $H \co (W,N_1,N_2,W_\del) \ra (M \times I,M \times
\{0\},M \times\{1\},\del M \times I)$, such that $H|_{N_1} = h_1$
and $H|_{N_2} = h_2$ and $H|_{W_\del}$ are homeomorphisms.

The (simple) structure set $\sS_\partial (M)$ is defined as the set
of equivalence classes of (simple) manifold structures on $M$
relative to $\del M$. In the case where $\del M$ is empty we write $\sS (M)$.
\end{defn}

More generally, assume that $(M, \del_1 M, \del_2 M)$ is a manifold
$3$-ad (see \cite[Chapter 0]{Wall(1999)}, note that one or both of
$\del_i M$ may be empty). A (simple) manifold structure on $M$
relative to $\del_1 M$ consists of a (simple) homotopy equivalence
of $n$-dimensional compact manifold $3$-ads
\[
(h, \partial_1 h, \partial_2 h) \co (N, \del_1 N, \del_2 N) \ra (M,
\del_1 M, \del_2 M)
\]
such that $\del_1 h$ is a homeomorphism. So one allows more
flexibility on the part of the boundary $\del_2 M$. There is a
corresponding equivalence relation which allows to one define the
(simple) structure set in this setting, which is denoted
$\sS_{\del_1 M} (M)$. Hence $\del_2 M$ does not appear in the
notation, which usually does not cause a confusion.

Also note that the $s$-cobordism theorem entails that, if $\dim (M)
= n \geq 5$, then two simple manifold structures $h_1$ and $h_2$ are
equivalent if and only if there exists a homeomorphism $f \co N_1
\ra N_2$ such that $h_2 \circ f \simeq h_1$ rel $\del$.

All the results of the present paper work equally well for structure
sets and simple structure sets. Thus in the familiar notation, the
reader may substitute $\sS^s(M)$ or $\sS^h(M)$ for $\sS(M)$ and its
variants throughout the paper.  To keep the language simple we will
work with structure sets and manifold structures.

The main tool for determining $\sS_\del (M)$ for a specific manifold
$M$, with $\dim (M) = n \geq 5$, is the surgery exact sequence (see
\cite[Chapter 10]{Wall(1999)}, \cite{Kirby-Siebenmann(1977)} for
definitions and details):
\begin{equation} \label{eqn:ses}
\cdots \ra L_{n+1} (\ZZ[\pi_1(M)]) \ra \sS_\del (M) \ra [M/\del M;
\G/\TOP] \ra L_n (\ZZ[\pi_1(M)]).
\end{equation}
We remark that the expression ``exact sequence'' makes sense, as
explained in \cite[Chapter 10]{Wall(1999)}, despite the fact that
the structure set, as defined, is only a pointed set with the base
point the identity $\id \co (M,\partial M) \ra (M,\partial M)$. On
the other hand, as pointed out in the introduction, one can endow
$\sS_\del (M)$ with the structure of an abelian group, which is
natural in the sense that the above sequence indeed becomes an exact
sequence of abelian groups. This follows from the identification of
the surgery exact sequence (\ref{eqn:ses}) with the algebraic
surgery exact sequence (\ref{alg-sur-seq-concrete}) which will be
discussed in detail in Section \ref{sec:alg-sur}.

In this section we only want to discuss the case when the compact
manifold in question is of the form $M \times D^k$, for $M$ closed
with $k \geq 1$. Then there is an easy geometric way of defining the
structure of a group on $\sS_\partial (M \times D^k)$ which is
abelian if $k \geq 2$. Abstractly, this follows from the observation
that $\sS_\partial (M \times D^k)$ is the $k$-th homotopy group of a
certain space, as explained for example in \cite[Chapter
17A]{Wall(1999)}. But we also need an explicit description of the
addition. For this denote
\begin{equation*}
S^{k-1}_\pm : =  \{ x = (x_1, \ldots , x_{k}) \in S^{k-1} \; | \; \pm x_1 \geq 0 \} \\
\end{equation*}
and note that each element in $\sS_\partial (M \times D^k)$ can be
represented as a homotopy equivalence of manifold triads
\begin{equation*}
(h, \partial_+ h, \partial_- h) \co (N, \del_+ N, \del_- N) \ra (M
\times D^k, M \times S_+^{k-1}, M \times S_-^{k-1})
\end{equation*}
where $\del_\pm h$ are homeomorphisms. Further denote
\begin{equation*}
D^{k}_\pm : =  \{ x = (x_1,\ldots ,x_{k}) \in D^{k} \; | \; \pm x_1 \geq 0 \} \\
\end{equation*}
and choose suitable homeomorphisms $(D^k, S^{k-1}_+, S^{k-1}_-)
\cong (D^k_+, S^{k-1}_+, D^{k-1})$ and $(D^k, S^{k-1}_+, S^{k-1}_-)
\cong (D^k_-, D^{k-1}, S^{k-1}_-)$. Also note $D^k = D^k_+
\cup_{D^{k-1}} D^k_-$.

\begin{defn} \label{defn:stacking}
Let $h_i \co N_i \ra M \times D^k$ with $i = 1,2$ be maps which
represent elements in  $\sS_\partial (M \times D^k)$. Define
$h_1+h_2 = h$ by
\begin{equation} \label{defn:addition}
h = h_1 \cup h_2 \co N = N_1 \cup_g N_2 \ra M \times D^k = M \times
D^k_+ \cup M \times D^k_-
\end{equation}
where $g \co \del_+ N_1 \ra \del_- N_2$ is given by $g = (\del_-
h_2)^{-1} \circ \del_+ h_1$.
\end{defn}

\subsection{Structure sets and the $\rho$-invariant}

Next we define the reduced $\rho$-invariant functions.

\begin{defn} \label{defn:reduced-rho}
Let $M$ be a closed oriented manifold of dimension $n = (2d-1) \geq
5$ with a reference map $\lambda (M) \co M \ra BG$ where $G$ is a
finite group. Define the function
\[
\wrho \co \sS (M) \ra \QQ \RhG^{(-1)^d} \quad \textup{by} \quad
\wrho ([h]) = \rho (N,\lambda (M) \circ h) - \rho (M,\lambda (M)),
\]
where the orientation on $N$ is chosen so that the homotopy
equivalence $h \co N \ra M$ is a map of degree $1$.
\end{defn}
Since the $\rho$-invariant is an $h$-cobordism invariant
\cite[Corollary 7.5]{Atiyah-Singer-III(1968)}, the function $\wrho$
is well-defined.

The definition in the relative setting comes from
\cite[Section 3]{Madsen-Rothenberg-II(1989)}. We need a little preparation.
Consider an element $[h]$ in $\sS_\partial (M \times D^k)$. Let
$M(h)$ be a closed manifold given by
\begin{equation} \label{defn:M(h)}
M (h) := N \cup_{\partial h} (M \times D^k).
\end{equation}
If $h$ is the identity we obtain $M (\id) \cong M \times S^k$, in
general the map $h$ induces $M(h) \simeq M \times S^k$, and if $k
\geq 2$ then $\pi_1(M(h)) \cong \pi_1 (M)$. When $M$ is oriented we
equip $N$ with an orientation so that $h$ is a map of degree $1$.
The orientation on the closed manifold $M(h)$ can then be chosen so
that it agrees with the given orientation on $N$ and it reverses the
orientation on $M \times D^k$. If $M$ possess a reference map
$\lambda (M) \co M \ra BG$ then we obtain a reference map $\lambda
(M(h)) \co M(h) \simeq M \times S^k \ra M \ra BG$.

\begin{defn} \label{defn:reduced-rho-del}
Let $M$ be a closed oriented manifold of dimension $n$ with a
reference map $\lambda (M) \co M \ra BG$ where $G$ is a finite group
and let $k \geq 1$ be such that $n+k = 2d-1 \geq 5$. Define the
function
\[
\wrho_\partial \co \sS_\partial (M \times D^k) \ra \QQ \RhG^{(-1)^d}
\quad \textup{by} \quad  \wrho_\partial ([h]) := \rho
(M(h),\lambda(M(h))).
\]
\end{defn}
Again this well-defined. Also notice that if $k \geq 1$ then
$\wrho_\partial ([\id]) = 0$.

Now we would like to understand the behaviour of $\wrho_\partial$
with respect to $+$ defined in \ref{defn:addition}. First a
definition and then an observation.
\begin{defn} \label{defn:ctd-sum-along-M}
Let $h_i \co N_i \ra M \times D^k$ with $i = 1,2$ be maps which
represent elements in  $\sS_\partial (M \times D^k)$. Consider $\bar
M (h_i) : = M(h_i) \smallsetminus \textup{int} (M \times D^{k-1}
\times [-\epsilon,\epsilon])$, for small $\epsilon > 0$, where
$D^{k-1} = D^k_+ \cap D^k_-$. Identify $\del \bar M(h_i)$ with $M
\times S^{k-1}$. Define the closed oriented manifold $M(h_1) \#_M
M(h_2)$ by:
\[
M(h_1) \#_M M(h_2) := \bar M (h_1) \cup_{\id \times r} \bar M (h_2)
\]
where $r$ is an orientation reversing homeomorphism of $S^{k-1}$.
\end{defn}

\begin{lem} \label{plus-vs-par-ctd-sum}
There is a homeomorphism of oriented manifolds
\[
M(h_1) \#_M M(h_2) \cong M(h_1+h_2).
\]
\end{lem}

\begin{proof}
Both sides can be identified with the union:
\[
N_1 \cup_{g_1} M \times S^{k-1} \times I \cup_{g_2} N_2
\]
where $g_1 = \del h_1 \co \del N_1 \ra M \times S^{k-1}$ and $g_2 =
r \circ \del h_2 \co \del N_2 \ra M \times S^{k-1}$.
\end{proof}

From the definition $\wrho_\del([h]) = \rho(M(h))$ we see that
Proposition \ref{prop-C} is equivalent to the following

\begin{prop} \label{rho-add-for-bdry}
There is an equality
\[
\rho (M(h_1+h_2)) = \rho (M(h_1)) + \rho (M(h_2)).
\]
\end{prop}

\begin{proof}
Let $Z(h_1)$ be such that $\partial Z(h_1) = k \cdot M(h_1)$ and let
$Z(h_2)$ be such that $\partial Z(h_2) = l \cdot M(h_2)$. Then
$\partial l \cdot Z(h_1) = kl \cdot M(h_1)$ and $\partial k Z(h_2) =
kl \cdot M(h_2)$, so we can assume $k=l$. In fact we will assume
$k=l=1$, which makes the notation simple, the easy generalisation is
left for the reader. Using Lemma \ref{plus-vs-par-ctd-sum} we build
a coboundary for $M(h_1+h_2)$ from the coboundaries $Z(h_1)$ and
$Z(h_2)$  by the following construction.

Note that the manifold $M(h_1)$, as a boundary component of $Z(h_1)$, has a
collar. Denote by $U(h_1) \subset Z(h_1)$ the portion of that collar
along $M \times D^{k-1} \times [-\epsilon,\epsilon] \subset M(h_1)$.
Construct the manifold $\bar Z(h_1)$ by removing $U(h_1)$ (the
interior and a suitable part of the boundary). Then the boundary of
$\bar Z(h_1)$ is decomposed as
\[
\bar M(h_1) \cup \partial' \bar Z(h_1) = \bar M(h_1)
\cup (M \times S^{k-1}) \times I \; \cup \; (M \times D^{k-1} \times [-\epsilon,\epsilon]) \times \{1\}.
\]
Similarly for $h_2$ instead of $h_1$. Recall also the orientation
reversing homeomorphism $r \co S^{k-1} \ra S^{k-1}$ which we can
extend to $r : D^k \ra D^k$ and identify $D^{k-1} \times
[-\epsilon,\epsilon] \cong D^k$. We define
\[
Z := \bar Z(h_1) \cup_{\bar r} \bar Z(h_2)
\]
where $\bar r \co \del' Z(h_1) \ra \del' Z(h_2)$ is the
homeomorphism given by
\[
\bar r = (\id \times r \times \id) \cup (\id \times r \times \{1\}).
\]
The following picture depicts the situation.

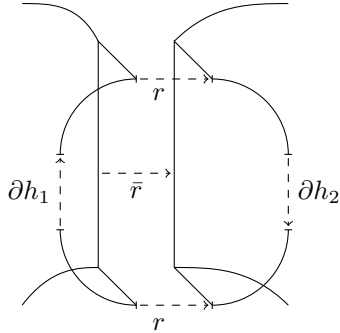
\[
\begin{tikzpicture}[scale=0.5]

\draw (3cm,1cm) arc (0:90:2cm);

\draw (-1cm,3cm) arc (90:180:2cm);

\draw (-3cm,-1cm) arc (180:270:2cm);

\draw (1cm,-3cm) arc (270:360:2cm);

\draw (2.9,1) -- (3.1,1);

\draw (1,2.9) -- (1,3.1);

\draw (-1,2.9) -- (-1,3.1);

\draw (-2.9,1) -- (-3.1,1);

\draw (-2.9,-1) -- (-3.1,-1);

\draw (-1,-2.9) -- (-1,-3.1);

\draw (1,-2.9) -- (1,-3.1);

\draw (2.9,-1) -- (3.1,-1);

\draw[dashed][->] (-3,-0.9) -- (-3,0.9);

\draw (-3,0) node[anchor=east] {$\partial h_1$};

\draw[dashed][->] (3,0.9) -- (3,-0.9);

\draw (3,0) node[anchor=west] {$\partial h_2$};

\draw[dashed][->] (-0.9,3) -- (0.9,3);

\draw (-0.4,2.2) node[anchor=south] {$r$};

\draw[dashed][->] (-0.9,-3) -- (0.9,-3);

\draw (-0.4,-3.1) node[anchor=north] {$r$};

\draw (-1,-3) -- (-2,-2) -- (-2,4) -- (-1,3);

\draw (1,-3) -- (0,-2) -- (0,4) -- (1,3);

\draw[dashed][->] (-1.9,0.5) -- (-0.1,0.5);

\draw (-1,0.5) node[anchor=north] {$\bar r$};

\draw (0,4) .. controls (1,5) and (2,5) .. (3,5);

\draw (0,-2) .. controls (1,-2) and (2,-2) .. (3,-3);

\draw (-2,4) .. controls (-2.5,5) and (-3.25,5) .. (-4,5);

\draw (-2,-2) .. controls (-2.5,-2) and (-3.25,-2) .. (-4,-3);

\end{tikzpicture}
\]

Since $Z$ is obtained by gluing $\bar Z(h_1)$ and $\bar Z(h_2)$
along $M \times D^k \subset \del \bar Z(h_i)$, $i = 1, 2$, we see that
by construction the boundary of $Z$ is $M(h_1+h_2) = M(h_1) \#_M
M(h_2)$.  To relate the $G$-signature of $Z$ to the $G$-signatures
of $Z(h_1)$ and $Z(h_2)$ we must take into account the signature defect
described by Wall \cite{Wall(1969)} which obstructs the additivity of
the signature.
If $\sign(Y)$ denotes the usual signature of a compact even-dimensional
manifold $Y$, then Wall's main Theorem, \cite[Theorem p.217]{Wall(1969)} states that
\[ \sign(Z) = \sign(Z(h_1)) + \sign(Z(h_2)) + \sign(V; A, B, C) \]
where $(V; A, B, C)$ denotes a certain bi-linear form defined by the
decomposition $Z = \bar Z(h_1) \cup_{\bar r} \bar Z(h_2)$.  In fact,
Wall concludes his paper by pointing out that his arguments computing the
signature defect work equally well for equivariant intersection forms over
a finite group.  Hence we have
\begin{equation} \label{eq:sig_defect}
 \Gsign(Z) = \Gsign(Z(h_1)) + \Gsign(Z(h_2)) + \Gsign(V; A, B, C).
\end{equation}
The key point is that in our setting the module $V$ on which $(V; A, B, C)$
is defined is zero and hence $\Gsign(V; A, B, C) = 0$.  To see this, we first
recall that for $i = 1, 2$ there is a homotopy equivalence of triples
\[ (M(h_1); N, M \times D^k_S) \simeq (M \times S^k; M \times D^k_N, M \times D^k_S) \]
where $S^k = D^k_N \cup_{S^{k-1}} D^k_S$.  It follows that each of the modules $A, B, C$ appearing in Wall's definition of $V$ is equal to the kernel of the inclusion
homomorphism
\[ K : = {\rm Ker} \left( H_{d-1}(M \times S^{k-1}) \to H_{d-1}(M \times D^{k}) \right), \]
where $2d-1 = n+k$ with $n = \dim (M)$. Hence
\[ V := \frac{A \cap (B+C)}{(A \cap B) + (A \cap C)} = \frac{K}{K} = 0 \]
and by \eqref{eq:sig_defect} we have $\Gsign(Z) = \Gsign(Z(h_1)) + \Gsign(Z(h_2))$.
From the definition of the $\rho$-invariant we have
\begin{align*}
\rho(M(h_1+h_2)) & = \Gsign (Z) \\
& = \Gsign (Z(h_1)) + \Gsign (Z(h_2)) \\
& = \rho (M(h_1)) + \rho (M(h_2)).
\end{align*}
\end{proof}
%

\begin{rem}
In an earlier version of this paper we stated that the last step of the proof of Proposition \ref{rho-add-for-bdry} followed from the Novikov additivity of the $G$-signature. However, in general Novikov additivity does not hold when gluing manifolds together along parts of their boundaries, there is a defect term. Nevertheless in the setting of our Proposition \ref{rho-add-for-bdry}, the defect term vanishes by the arguments presented above, so we indeed have additivity of the $G$-signature.
\end{rem}


\section{MAFs and MCNs} \label{sec:mafs-and-mcns}


This section contains preparatory material about mapping cylinder
neighbourhoods which will be used in the construction of the map
$CW^j$ in the following section.

Let $X^n \subset Y^{n+q}$, $q \geq 1$, be a locally flat
submanifold. A {\it mapping cylinder neighbourhood}, MCN, of $X$ is a
codimension-$0$ submanifold with boundary of $Y$, $(N,\del N)$, such
that $X \subset \int (N)$ and there is a deformation retraction $p
\co N \ra X$ such that $\del p := p|_{\del N} \co \del N \ra X$
satisfies $(\cyl (\del p),\del N) \cong (N,\del N)$, where $\cyl
(\del p))$ denotes the mapping cylinder of $\del p$. See
\cite[Section 3]{Quinn(1979)} for more information about the existence of MCNs.

Theorem \ref{thm:mafs-vs-mcns} below, which is taken from
\cite{Hughes(1999)}, recalls a characterisation of MCNs using
manifold approximate fibrations (MAFs). An {\it approximate
fibration} $p \co P \ra N$ is a map which has an approximate
homotopy lifting property. It is called a MAF if both $P$ and $N$
are manifolds. For more information we refer the reader to
\cite[Section 1]{Hughes-Taylor-Williams(1990)}.

We will need the following two properties. Firstly, it follows
easily from the definitions that a composition of MAFs is a MAF.
Secondly, being a MAF is a local property: this means that in order to determine
whether a map $p \co P \ra N$ between closed manifold is a MAF it is
enough to check this property in a neighbourhood of each point of
$N$, see \cite[Corollary 12.14]{Hughes-Taylor-Williams(1990)} and
\cite[Proposition 2.2]{Chapman(1980)}.

\begin{thm}{\cite[Theorem 6.1]{Hughes(1999)}}  \label{thm:mafs-vs-mcns}
Let $p \co P \ra N$ be a map between closed manifolds with $\dim
(P) = m$ and $\dim (N) = n$. If $\cyl (p)$ is a manifold with $N$ a
locally flat submanifold then $p$ is a MAF with $\hofib (p) \simeq
S^{m-n}$. The converse is also true if $m \geq 5$.
\end{thm}

\begin{rem}
The above theorem is stated in \cite{Hughes(1999)} with the
dimension restriction for both implications. However, it is clear
from the proof that it is used only in one direction.
\end{rem}

We need a generalisation of the above statement for compact manifolds
with boundary. Let $(p,\del p) \co (P,\del P) \ra (N,\del N)$ be a
map of compact manifolds with boundary. If $\del N$ has more than
one connected component, we index these by $\alpha \in \sI$ and
denote by $\del_\alpha p \co \del_\alpha P \ra \del_\alpha N$ the
corresponding restriction of $\del p$. The mapping cylinder $\cyl
(\del p)$ is a subspace of the mapping cylinder $\cyl (p)$ and we
have
\[
\del \cyl (p) = \cyl (\del p) \cup_{\del P \times \{0\}} P \times \{0\}.
\]
The triple $(\cyl (p); \cyl (\del p), P \times \{0\})$ defines a
triad of spaces.

\begin{cor} \label{cor:relative-mafs-vs-mcns}
Let $(p,\del p) \co (P,\del P) \ra (N,\del N)$ be a map of compact
manifolds with boundary, $\dim (P) = m$, $\dim (N) = n$. Assume in
addition that on a collar of $\del P$ the map $p$ is the product map
of $\del p$ with the identity in the collar direction. If the triad
$(\cyl (p); \cyl (\del p), P \times \{0\})$ is a manifold triad with
$(N,\del N)$ a locally flat submanifold (with boundary) of $(\cyl
(p),\cyl (\del p))$ then $(p,\del p)$ is a pair of MAFs and we have
for each connected component $\alpha \in \sI$
\[
\hofib (\del_\alpha p) \simeq \hofib (p) \simeq S^{m-n}.
\]
The converse is also true if $m \geq 6$.
\end{cor}

By a pair of MAFs we just mean that both $\del p$ and $p$ are MAFs.
For a manifold triad see \cite[Chapter 0]{Wall(1999)}.

\begin{proof}
The following two observations are used. Firstly, the fact recalled
above that being a MAF is a local property. Secondly, for a map $p
\co P^m \ra N^n$ between closed manifold, we have that $\hofib (p)
\simeq S^{m-n}$ if and only if $(\cyl (p),P)$ is a Poincar\'e pair.
Also, for $(p,\del p) \co (P,\del P) \ra (N,\del N)$, a map between
compact manifolds with boundary, we have that $\hofib (\del_\alpha
p) \simeq \hofib (p) \simeq S^{m-n}$ for each $\alpha \in \sI$ if
and only if $(\cyl (p);P,\cyl (\del p))$ is a Poincar\'e triad.

The general idea is to reduce the proof to considerations about the
map
\[
\bar p := p \cup_{\del p} p \co (P \cup_{\del P} P) \longrightarrow (N \cup_{\del N}
N).
\]

To prove the if part note that $\bar p$ fulfils the assumptions of
the if part of Theorem \ref{thm:mafs-vs-mcns}. It is a MAF because
being a MAF is a local property and both parts are MAFs. We have $\hofib
(\bar p) \simeq S^{m-n}$, since $(\cyl (\bar p),P)$ is a Poincar\'e
pair, because it is obtained by gluing two Poincar\'e triads. Hence
by Theorem \ref{thm:mafs-vs-mcns}, $\cyl (\bar p)$ is a closed
manifold with $N \cup N$ locally flatly embedded. But this implies
that $(\cyl (p),\cyl(\del p) \subset \cyl (\bar p)$ is a codimension-$0$ submanifold with boundary with $(N,\del N)$ locally flatly
embedded.

To prove the only if part observe immediately that $\cyl (\bar p)$
fulfils the assumptions of Theorem \ref{thm:mafs-vs-mcns}. Hence
$\bar p$ is a MAF with $\hofib (\bar p) \simeq S^{m-n}$. We obtain
that $(p,\del p)$ is a MAF pair by the locality of the MAF condition.
The statement about the homotopy fibre follows immediately from the
assumptions even without invoking $\bar p$.
\end{proof}


Recall from surgery theory that the normal invariants of a manifold
$N$, which means the bordism set of degree one normal maps with
target $N$, are in one-to-one correspondence with $[N,\G/\TOP]$,
homotopy classes of maps from $N$ to the space $\G/\TOP$. If instead of $N$ we consider the manifold with boundary $N \times D^{r+1}$ and if $r \geq 2$, then we have a
$(\pi\!-\!\pi)-$situation and hence
$$\sS (N \times D^{r+1},\del) \cong [N \times D^{r+1} ; \G/\TOP] \cong [N;\G/\TOP].$$
Therefore we can ``realise'' elements of $[N;\G/\TOP]$ as homotopy equivalences
of manifolds with boundary $f \co (W,\del W) \ra (N \times
D^{r+1},\del)$. As such $(W,\del W)$ is just some manifold with
boundary homotopy equivalent to $N \times D^{r+1}$. The following
proposition says that there is more structure in this situation: the
manifold $M$ can be identified as a MCN of $N$ and the map $f$ as
the cylinder of the restriction of $f$ to the boundary.

\begin{prop} \label{prop:constructing-neighborhoods}
Let $r \geq 2$, let $n + r \geq 5$ and let $N$ be a closed manifold of
dimension $n$. For any element $\chi \in [N,\G/\TOP]$ there exists a
commutative diagram
\[
\xymatrix{
(W,\del) \ar[d]_{p} \ar[r]^(0.4){\omega} & (N \times D^{r+1},\del) \ar[d]^{\pr_1} \\
N \ar[r]_{\id} & N }
\]
where $p \co W \ra N$ is a MCN with $N$ a locally flat submanifold and
the map $\omega \co W \ra N \times D^{r+1}$ is induced from a fibre
homotopy equivalence of $r$-dimensional spherical fibrations
associated to $p$ and $\pr_1$.  In particular $\omega$ is a homotopy
equivalence of pairs such that $[\omega] = \chi \in \sS (N \times
D^{r+1},\del) \cong  [N;\G/\TOP]$.
\end{prop}

\begin{proof}
This is proved by Hutt in \cite[Section 1]{Hutt(1998)}. He closely
follows Pedersen in \cite[Proof of Lemma 9]{Pedersen(1975)} with
some additional ingredients from \cite{Quinn(1979)}. Pedersen's
arguments are in turn based on \cite{Rourke-Sanderson(1970)}.

In \cite[Section 0]{Rourke-Sanderson(1970)} a simplicial group ${\bf
T}op_{r+1}$ is defined with the following properties: (1) the set of
equivalence classes of germs of codimension $(r+1)$ topological
neighbourhoods of a topological manifold $N$ is in bijective
correspondence with the set of homotopy classes of maps $[N, B{\bf
T}op_{r+1}]$, (2) there is a map ${\bf T}op_{r+1} \to G_{r+1}$, the
group of self-equivalences of $S^r$, such that the inclusions
$G_{r+1} \to \G$ and ${\bf T}op_{r+1} \to \TOP$ give rise to a
homotopy equivalence $G_{r+1}/{\bf T}op_{r+1} \simeq \G/\TOP$.

Let $x_{r+1} : N \to G_{r+1}/{\bf T}op_{r+1}$ be a map representing
$\chi_{r+1} \in [N, G_{r+1}/{\bf T}op_{r+1}]$ where $\chi_{r+1}$
corresponds to $\chi$ under the above equivalence.  If $i:
G_{r+1}/{\bf T}op_{r+1} \to B{\bf T}op_{r+1}$ is the canonical map
then $i \circ x_{r+1} : N \to B{\bf T}op_{r+1}$ defines the germ of
a codimension-$(r+1)$ topological neighbourhood of $N$.  Let $V
\supset N$ be such a neighbourhood.  By \cite[Theorem 3.1.1]{Quinn(1979)} there is
a MCN of $N$ in $V$, $W \supset N$ and we let $p : \del W \to N$ be
the MAF associated to $W \supset N$.  The homotopy fibre of $p$ is
$S^r$ and we convert $p$ into a spherical fibration $S(p) : S(\del
W) \to N$.  So far we have only used $i \circ x_{r+1}$: by
definition the map $x_{r+1}$ defines a fibre homotopy trivialisation
of $S(p)$ and so we obtain a homotopy equivalence $\tau : \del W
\simeq N \times S^r$.  We set $\omega := \cyl(\tau): (W, \del W)
\simeq (N \times D^{r+1}, \del)$.

It remains to show that the normal invariant of $\omega$ is $\chi$.
As discussed above, we have a canonical identification of the normal
invariant set of $N \times D^{r+1}$ with $[N, \G/\TOP]$.  The normal
invariant of any degree one normal map $f: (M, \del M) \to (N \times
D^{r+1}, \del)$ can be computed from the normal invariant of the
splitting obstruction along $N \times \{ 0 \} \subset N \times
D^{r+1}$.  But by definition, the splitting obstruction to $\omega$
along $N \times 0$ has normal invariant $\chi$: see
\cite[Theorem 2.23]{Madsen-Milgram(1979)} for this identification
which was stated there for the $PL$-category but remains true in the
topological category given the later proof of topological
transversality \cite{Kirby-Siebenmann(1977)}, \cite[Ch. 9]{Freedman-Quinn(1990)}.
\end{proof}

We will also need a relative version of Proposition
\ref{prop:constructing-neighborhoods}. In that case we are given a
manifold with boundary $(N,\del N)$ of dimension $n$. The product $N
\times D^{r+1}$ becomes a manifold triad and we define $\del_0 (N
\times D^{r+1}) := N \times S^r$ and $\del_1 (N \times D^{r+1}) :=
\del N \times D^{r+1}$. Then $\del \del_1 = \del \del_0 = \del N
\times S^r$. If $r \geq 2$ we are again in a
$(\pi\!-\!\pi)-$situation and hence $\sS (N \times
D^{r+1},\del_0,\del_1) \cong [N \times D^{r+1} ; \G/\TOP] \cong
[N;\G/\TOP]$.

\begin{prop} \label{prop:constructing-neighborhoods-relative}
Let $r \geq 2$, $n + r \geq 5$ and let $(N,\del N)$ be a compact
manifold of dimension $n$ with boundary. Suppose given a MCN
$(\del_1 W, \del \del_1 W) \ra \del  N$, with $\del N$ a locally flat
submanifold, whose associated $r$-spherical fibration is fibre
homotopically trivialised by a map $(\del_1 W, \del \del_1 W) \ra
(\del N \times D^{r+1},\del N \times S^r)$ which is represented by a
map $\xi \co \del N \ra \G/\TOP$. For any element $\chi \in
[N,\G/\TOP]$, such that $\chi|_{\del N} = \xi$ there exists a
diagram
\[
\xymatrix{
(W,\del) \ar[d]_{p} \ar[r]^(0.4){\omega} & (N \times D^{r+1},\del) \ar[d]^{\pr_1} \\
N \ar[r]_{\id} & N }
\]
where $p \co W \ra N$ is a MCN with $N$ a locally flat submanifold and
the map $\omega \co W \ra N \times D^{r+1}$ is induced from a fibre
homotopy equivalence of $r$-dimensional spherical fibrations
associated to $p$ and $\pr_1$.  In particular $\omega$ is a homotopy
equivalence of pairs, and restrictions of everything to the
appropriate parts of boundary agree with the given structures, such
that $[\omega] = \chi$.
\end{prop}

\begin{proof}
The theorems of \cite{Rourke-Sanderson(1970)} and \cite{Quinn(1979)}
and \cite{Madsen-Milgram(1979)} used in the proof of Proposition
\ref{prop:constructing-neighborhoods} have suitable relative
versions. The proof is then a routine modification of the arguments in Proposition \ref{prop:constructing-neighborhoods}.
\end{proof}

\begin{rem}
In Proposition \ref{prop:constructing-neighborhoods-relative} we had
another possibility: to start with a map $N \ra \G/\TOP$ without
specifying the MCN over the boundary. We could then use the absolute
version to produce such a MCN and then use the relative version. In
that case, however, the dimension restrictions would have to be
relaxed by $1$ which would be inconvenient later.
\end{rem}


\section{The Cappell-Weinberger map} \label{sec:cw-map}


In this section we recall Hutt's description of the
Cappell-Weinberger map and prove some basic facts about this map. We
present the map both for the usual $4$-periodicity and also for
$8$-periodicity: a possibility pointed out in
\cite[p. 48]{Cappell-Weinberger(1985)}.  Thus let $\F = \C$ or $\h$ and let
$k = 2$ or $4$ be the dimension of $\F$ over $\R$.

Let $h : N \to M$ be a homotopy equivalence of closed topological
manifolds of dimension $n \geq 5$ representing $[h] \in
\mathcal{S}(M)$. From $h$ the Hutt construction produces a homotopy
equivalence $h' \co N' \ra M \times D^{2k}$ of manifolds with
boundary defined by (\ref{def:N-prime}) and (\ref{def:h-prime})
below. The restriction of $h'$ to the boundary of $N'$ is a
homeomorphism and so $h'$ represents an element $[h'] \in \sS_\del
(M \times D^{2k})$. The mapping of structures, $[h] \mapsto [h']$,
is the $CW$-map of Definition \ref{def:CW-map}. The rest of the
first subsection is devoted to proving that this map is well
defined. In the second subsection we will review the construction of
a homotopy equivalence of closed manifolds $\widehat h \co \widehat
N \ra M \times \FF P^2$ given by ``extension by a homeomorphism''
from $h'$. In Lemma \ref{lem:phi} we will show that this structure
is equivalent in $\sS(M \times \FF P^2)$ to another structure, $\bar
h \co \bar N \ra M \times \FF P^2$ which has a certain
factorisation, whereas $h'$ does not possess an analogous
factorisation.

It is not immediately clear that the Hutt construction produces a
periodicity map.  We prove this later in Section \ref{sec:sieb-per}
for $\FF = \HH$. We note that an essential component of the
construction is the use of certain $S^{k-1}$-branched coverings.
This permits the extension by a homeomorphism mentioned above, which
is a key ingredient in the proof of the fact that the $CW$-map is a
periodicity map.

It is useful to observe that the constructions of $h'$, $\widehat h$
and $\bar h$ have two components.  It is one thing to construct the
sources $N'$, $\widehat N$ and $\bar N$ of the above maps and it is
another issue to construct the maps to $M \times D^{2k}$ and $M
\times \F P^2$. In particular it is easier to construct the sources.
We use this point in the last Section \ref{sec:completion}, where we
apply a version of the Hutt construction for manifolds with boundary
to construct a coboundary for the closed manifold $\bar N$.

\subsection{Definition} Let $h : N \to M$ be a homotopy equivalence of closed topological
manifolds of dimension $n \geq 5$ representing $[h] \in
\mathcal{S}(M)$. Our starting point is the following
\begin{lem}{\cite[Section 1]{Hutt(1998)}} \label{lem:hutt-1}
There is a commutative diagram of maps of pairs
\[
\begin{diagram}
\node{(\bar W, \del \bar W)} \arrow{e,t}{\omega} \arrow{s,r}{}
\node{(N \times D^{k+1}, N \times S^k)}  \arrow{e,t}{h \times {\id}}
\arrow{s,r}{} \node{(M \times D^{k+1}, M \times S^k)} \arrow{s,r}{}
\\ \node{N} \arrow{e,t}{{\id}} \node{N} \arrow{e,t}{h} \node{M}
\end{diagram}
\]
where $\bar W$ is a mapping cylinder neighbourhood of $N$ in which
$N$ is locally flatly embedded and $\omega : \bar W \to N \times
D^{k+1}$ is a homotopy equivalence of pairs such that the composite
$\bar \psi : = (h \times {\id}) \circ \omega$ is $h$-cobordant, as a
map of pairs, to a homeomorphism.
\end{lem}

\begin{proof}
The homotopy equivalence $h \times {\id}$ defines an element of
$\mathcal{S}(M \times D^{k+1})$ which, by the
$(\pi\!-\!\pi)-$theorem \cite[Chapter 4]{Wall(1999)}, is isomorphic
to the normal invariant set $\mathcal{N}(M \times D^{k+1}) \cong
[M,\G/\TOP]$.  Hence $h \times {\id}$ defines an element $[h \times
{\id}] \in [M,\G/\TOP]$.  We choose $\chi \in [N,\G/\TOP]$ so that
$(h^{-1} \times {\id})^\ast (\chi) = - [h \times {\id}]$: here we
take the negative with the Whitney sum group structure on $\G/\TOP$
and note that since $(h^{-1})^ \ast$ induces an isomorphism of $[N,
\G/\TOP] \cong [M, \G/\TOP]$, such a $\chi$ exists.  By Proposition
\ref{prop:constructing-neighborhoods} the element $\chi \in
[N,\G/\TOP]$ gives rise to a homotopy equivalence of pairs
\[
\omega : (\bar W, \del \bar W) \to (N \times D^{k+1}, N \times S^k)
\]
with the claimed properties: $\bar W$ is a mapping cylinder
neighbourhood of $N$ in which $N$ is locally flatly embedded and
$\omega(N) = N \times \{0\} \subset N \times D^{k+1}$.

Finally, by the composition formula of \cite[Proposition 2.2]{Brumfiel(1971)} and
\cite[Lemma 2.5]{Madsen-Taylor-Williams(1980)} $(h \times {\id})
\circ \omega$ has trivial normal invariant and hence $\bar \psi : =
(h \times {\id}) \circ \omega$ is h-cobordant to a homeomorphism as
required.
\end{proof}

Lemma \ref{lem:hutt-1} states in particular that there is a homotopy
equivalence
\begin{equation} \label{map:H}
H: (U;\bar W,W,U_\del) \to (M \times D^{k+1} \times
[0,1];\{0\},\{1\}, M \times S^k \times [0,1])
\end{equation}
where $U$ is an $h$-cobordism of manifolds with boundary; we denote
$\del U = \bar W \cup W \cup U_\del$, where $U_\del$ is an
$h$-cobordism between $\del \bar W$ and $\del W$. Also $H|_{\bar W}
= \bar \psi$ and $H|_{W}$ is a homeomorphism $H|_{W} =: \psi : (W,
\del W) \cong (M \times D^{k+1}, M \times S^{k})$. For later use, we
write $H_\del : U_\del \to M \times S^{k} \times [0,1]$ for the
restriction of $H$ to the boundary part $U_\del$.  Note that this
differs from Hutt who only uses $H_\del$ and calls it $H$. As Hutt
observes, $\bar \psi$ and $\psi$ are maps with contrasting
properties: $\bar \psi$ is a map over $h$ which is not in general a
homeomorphism whereas $\psi$ is a homeomorphism but not in general a
map over $h$.\footnote{In fact we do not even have a preferred map
from $W$ to $N$. We could use the $h$-cobordism $U$ to obtain some
map, but still the map $\psi$ would not be a map over $h$ in
general.} We shall need to exploit these two properties in the
following construction and therefore slide between them via the map
$H$.

The key part of the construction is the following
``$S^{k-1}$-branched cover plus $h$-cobordism construction'' of
\cite[Section 1]{Cappell-Weinberger(1985)}. Let $\gamma$ denote the Hopf bundle
$S^{2k-1} \to S^{k}$ and for a manifold $X$ let $\gamma_X := \id_X
\times \gamma \co X \times S^{2k-1} \ra X \times S^k$. Further let
$\del \psi = \psi|_{\del W}$, $\del \bar \psi = \bar \psi|_{\del
\bar W}$ and $\del \omega = \omega|_{\del \bar W}$.  Define
\begin{equation} \label{defn:del-W-prime}
\del \bar W' := (\del \bar \psi)^*(\gamma_M) \cong (\del \omega)^*
(\gamma_N), \, \,  \del W' := (\del \psi)^*(\gamma_M), \, \,  U'_\del
:= (H_\del)^*(\gamma_{M \times [0,1]})
\end{equation}
to be the pull-backs.\footnote{The notation should be understood as
$(\del W)'$, not $\del (W')$, in fact there is no manifold $W'$.}
Observe that the manifold $U'_\del$ is an $h$-cobordism between
$\del \bar W'$ and $\del W'$, since $U_\del$ was an $h$-cobordism
between $\del \bar W$ and $\del W$. Further denote
\[
\del \bar \psi' \co \del \bar W' \ra M \times S^{2k-1}, \quad \del \psi' \co \del W' \ra M \times S^{2k-1}, \quad H'_\del \co U'_\del \ra M \times S^{2k-1} \times [0,1]
\]
the maps covering $\del \bar \psi$, $\del \psi$, and $H_\del$
respectively. The $S^{k-1}$-bundle projections are denoted by
\[
q_{\bar W} \co \del \bar W' \ra \del \bar W, \quad q_{W} \co \del W'
\ra \del W \quad \text{and} \quad q_{U} \co U'_\del \ra U_\del.
\]
By Theorem \ref{thm:mafs-vs-mcns} the map $p \co \del \bar W \to N$
is a MAF over $N$ with homotopy fibre $S^k$.  Obviously $q_{\bar W}$
is also a MAF.  As the composition of MAFs is a MAF, the map
\begin{equation} \label{defn:p-prime}
p' \co \del \bar W' \xra{q_{\bar W}} \del \bar W \xra{p} N
\end{equation}
is a MAF with homotopy fibre $S^{2k-1}$. Again applying Theorem
\ref{thm:mafs-vs-mcns} it follows that $\cyl (p')$, the mapping
cylinder of $p'$, is a MCN of $N$ with $N$ locally flatly embedded.
Denote
\begin{equation} \label{def:N-prime}
N' := \cyl (p') \cup U_\del'.
\end{equation}
The topological manifold $N'$, with boundary $\del W'$, will be the
domain of the structure used to define the Cappell-Weinberger map. The $S^{k-1}$-branched cover refers to the projection map $\cyl (p') \ra \cyl (p)$ which can be
viewed as such, the branching subset being $N \subset \cyl (p)$.

Next we define a homotopy equivalence $h' : N' \to M \times D^{2k}$
whose restriction to the boundary $h'|_{\del W'} :  \del W' \ra M
\times S^{2k-1}$ is a homeomorphism. We regard $M \times D^{2k}$ as
the union $(M \times D^{2k}_1)\cup_{M \times S^{2k-1} \times \{ 1
\}} (M \times S^{2k-1} \times [1,2]) $ where $D^{2k}_1$ has radius
$1$ and we re-parametrise $H'_\del \co U'_\del \ra M \times S^{2k-1}
\times [1,2]$.  Now define
\begin{equation} \label{def:h-prime}
h' := \cyl (\del \bar \psi',h) \cup H'_\del : N' \to M \times D^{2k}.
\end{equation}
We remark again that $h'$ is not a map over $h$, since the map $\del
\psi$, and hence the maps $\del \psi'$, and $H'_\del$, are not maps
over $h$.
\begin{defn} \label{def:CW-map}
For $k=2$ or $4$ the Cappell-Weinberger map is the map
\[ CW^{k/2} : \mathcal{S}(M) \to \mathcal{S}_\del(M \times D^{2k}), ~~~[h] \longmapsto [h'].\]
\end{defn}
\begin{rem}
The above construction is in fact quite subtle. Note that the map
$\bar \psi \cup H_\del$ is $h$-cobordant rel $\del$ to a
homeomorphism, hence representing the trivial element in $\sS_\del
(M \times D^{k+1})$. On the other hand, this argument cannot be used
for $h'$. The point is that $N'$ cannot be identified with the pull
back of $M \times D^{2k} \ra M \times D^{k+1}$ along $\bar \psi$. If
we wanted to have such an identification we would need $\bar \psi$
to be transverse to $M = M \times \{0\} \in  M \times D^{k+1}$. And
$\bar \psi$ is in general not transverse to this submanifold
(despite the equality $(\bar \psi)^{-1} (M) = N \subset \bar W$). If this were
the case, then we would immediately obtain that $h$ is normally
cobordant to a homeomorphism, which is not the case in general.
\end{rem}

The remainder of this subsection is devoted to proving that the
structure invariant $[h']$ is independent of all the choices made
during its construction from the structure invariant $[h]$.  In turn
these are:
\begin{enumerate}
\item the choice of $h : N \to M$ to represent $[h] \in \mathcal{S}(M)$,
\item the choice of the homotopy equivalence $\omega \co (\bar W, \del \bar W) \to (N \times D^{k+1}, N \times S^k)$ representing an element in $\sS (N \times D^{k+1})$,
\item the choice of the $h$-cobordism $U$, between $\bar W$ and some manifold $W$, with the homotopy equivalence $H \co U \ra M \times D^{k+1} \times [0,1]$ such that $H|_{\bar W} = \bar \psi$ and $H|_{W} = \psi$ which is some homeomorphism.
\end{enumerate}
Let $(h,\omega,H)'$ be the structure on $M \times D^{2k}$ produced
from a choice of $h$, $\omega$ and $H$.

If two homotopy equivalences $h_i \co N_i \ra M$, $i=0,1$, represent
the same element in $\sS (M)$, then there is an $h$-cobordism $(Z;
N_0,N_1)$ and a homotopy equivalence $h_Z \co Z \ra M \times [0,1]$
with restrictions $h_Z|_{N_i} = h_i$. The constructions of Lemma
\ref{lem:hutt-1} can be applied in the relative setting. This means
that  $h_Z$ can be precomposed with a homotopy equivalence $\omega_Z
\co \bar X \ra Z \times D^{k+1}$ from some $h$-cobordism $(\bar X;
\bar W_0, \bar W_1)$, which is also a MCN of $(Z; N_0,N_1)$ and we
obtain a homotopy equivalence
\begin {equation} \label{map:bar-Psi}
\bar \Psi \co \bar X \ra M \times D^{k+1} \times [0,1]
\end{equation}
which restricts to $\bar \Psi|_{\bar W_i} = \bar \psi_i$ for some
homotopy equivalences $\bar \psi_i$ which are $h$-cobordant to
homeomorphisms $\psi_i$.

The choice of $\omega$ produces a similar outcome. For $i=0,1$, two
homotopy equivalences $\omega_{ij} : \bar W_{ij} \to N_i \times
D^{k+1}$, for $j=0,1$, as in Lemma \ref{lem:hutt-1} represent the
same element of $\sS (N_i \times D^{k+1})$, so again there is an
$h$-cobordism $(\bar X_i, \bar W_{i0}, \bar W_{i1})$, which is also
a MCN of $(N_i \times [0,1],N_i \times \{0\},N_i \times \{1\})$ and also
a homotopy equivalence $\omega_{N_i \times D^{k+1}} \co \bar X_i \ra
N_i \times [0,1] \times D^{k+1}$. Composing with $h_i \times \id \co
N_i \times [0,1] \times D^{k+1} \ra M \times [0,1] \times D^{k+1}$
we obtain $\bar \Psi_i$ with analogous properties as the map $\bar
\Psi$ in (\ref{map:bar-Psi}). We can glue $\bar X$, $\bar X_0$ and
$\bar X_1$ along their common boundary components. There is a
corresponding homotopy equivalence to $M \times D^{k+1} \times
[0,1]$. This has precisely the same properties as $\bar \Psi$ in
(\ref{map:bar-Psi}), so we may assume from now on that $\bar \Psi$
represents the difference between choices ($h_0$, $\omega_0$) and
($h_1$, $\omega_1$).

Now we come to the choice of $H$. Consider a pair of $h$-cobordisms
$U_i$, $i = 0,1$, each from $\bar W_i$ to $W_i$, with homotopy
equivalences $H_i$ between $\bar \psi_i$ and homeomorphisms
$\psi_i$. We can glue these two $h$-cobordisms onto $\bar X$ along
$\bar W_0$ and $\bar W_1$, take a product of the result with $[0,1]$
and rearrange the boundary to obtain a compact manifold $Y$ with
boundary.   The manifold $Y$ can be seen as an $h$-cobordism either
between $h$-cobordisms $U_0$ and $U_1$ or between the $h$-cobordism
$\bar X$ and some $h$-cobordism $(X; W_0, W_1)$. We also obtain a
homotopy equivalence
\[
G : (Y;\bar X, X) \to (M \times D^{k+1} \times [0,1] \times [0,1],
[0,1] \times \{0\}, [0,1] \times \{1\})
\]
restricting to $\bar \Psi$ on the $h$-cobordism $(\bar X; \bar W_0,
\bar W_1)$ and to a homotopy equivalence $\Psi : X \to M \times
D^{k+1} \times [0,1]$ on $(X; W_0, W_1)$ such that $\Psi|_{W_i} =
\psi_i$, which are homeomorphisms. Further $G|_{U_i} = H_i$.

But we may also view $Y$ as an $h$-cobordism between $U_0 \cup X$
and $U_1$ (just thicken $W_1$ to $W_1 \times [0,1]$ and rearrange
the boundary again). The point is that under this change of
viewpoint the homotopy equivalence $G$ is a homeomorphism on $W_1
\times [0,1]$.

Recall the constructions between Lemma \ref{lem:hutt-1} and
Definition \ref{def:CW-map} which from the homotopy equivalence $H$
of (\ref{map:H}) produce the rel $\del$ structure $h'$ on $M \times
D^{2k}$ of (\ref{def:h-prime}).  Using Corollary
\ref{cor:relative-mafs-vs-mcns} in place of Theorem
\ref{thm:mafs-vs-mcns} we may now perform the precisely analogous
constructions with the homotopy equivalence $G$ to obtain a rel
$\del D^{2k} \times [0,1]$ structure on $M \times D^{2k} \times
[0,1]$.  This structure restricts to $(h_0, \omega_0, H_0 \cup
G|_{X})'$ and $(h_1, \omega_1, H_1)'$ on the respective ends and
hence gives an $h$-cobordism rel $\del D^{2k} \times [0,1]$ between
these two rel $\del D^{2k}$ structures on $M \times D^{2k}$.

(Here the role of $U$ in $(\ref{map:H})$ is played by Y, that of
$\bar W$ by $\bar X$, that of $W$ by $W_1 \times [0,1]$ and that of
$U_\del$ by the part of the boundary of $Y$ which constitutes the
$h$-cobordism between $(U_0)_\del \cup X_\del$ and $(U_1)_\del$
where $X_\del$ is the part of the boundary of $X$ which is the
$h$-cobordism between $\del W_0$ and $\del W_1$. Keep in mind that
$(\bar X; \bar W_0, \bar W_1)$ is a MCN of $(Z; N_0, N_1)$.)

Thus we have proved
\begin{lem} \label{choice1lem}
With $G$ as above $[(h_0, \omega_0, H_0 \cup G|_{X})'] = [(h_1,
\omega_1, H_1)'] \in \sS_\del (M \times D^{2k})$.
\end{lem}

It remains to show $[(h_0,\omega_0,H_0)'] = [(h_0, \omega_0, H_0
\cup G|_{X})']$ which we prove in the following paragraphs.

Consider now $G|_{X} \co X \ra M \times D^{k+1} \times [0,1]$ which
is an h-cobordism between the homeomorphisms $\psi_0$ and $\psi_1$.
As such $G|_{X}$ defines an element of $\mathcal{S}_{\del \{0,
1\}}(M \times D^{k+1} \times [0, 1])$ where the subscript $\del \{0,
1 \}$ indicates that all structure invariants are defined relative
to $M \times D^{k+1} \times \{ 0, 1\}$. Using the $S^{k-1}$-branched
cover viewpoint we will next show that from $G|_{X}$ we can obtain a
structure from $M \times D^{2k} \times [0, 1]$ relative to $M \times
D^{2k} \times \{ 0, 1\}$ which relates $ (h_0,\omega_0,H_0)'$ and
$(h_0, \omega_0, H_0 \cup G|_{X})'$.

The Hopf fibration $S^{2k-1} \to S^k$ is given by a free
$S^{k-1}$-action on $S^{2k-1}$.  If we take the cone of this action
we obtain an $S^{k-1}$ action on $D^{2k}$, free except at the centre
point, which exhibits $D^{2k}$ as a branched $S^{k-1}$-fibration
over $D^{k+1} = D^{2k}/S^{k-1}$. Taking the product with $M$ we have
$M \times D^{2k} \to M \times D^{k+1}$ which is a branched
$S^{k-1}$-fibration with branch set $M \times \{ 0 \} \subset M
\times D^{k+1}$.  Now let $f : X \to M \times D^{k+1} \times [0, 1]$
represent $[f] \in \mathcal{S}_{\del \{ 0, 1\}}(M \times D^{k+1}
\times [0, 1])$. If we make $f$ transverse to $M \times \{ 0 \}
\times [0, 1]$, then we may pull back the branched
$S^{k-1}$-fibration $\Gamma : M \times D^{2k} \times [0, 1] \to M
\times D^{k+1} \times [0, 1]$ along $f$. \footnote{Note that $f$ is
by definition a homeomorphism on $\del \{0, 1 \}$, hence transverse
to anything, hence does not have to be changed on $\del \{0, 1 \}$.}
The outcome, $f^*(\Gamma)$, is a branched $S^{k-1}$-fibration over
$X$ which defines a structure on $M \times D^{2k} \times [0, 1]$
which is relative to $M \times D^{2k} \times \{0, 1\}$. Using
transversality along the $h$-cobordisms which define the equivalence
relation in $\mathcal{S}_{\del \{0, 1\}}(M \times D^{k+1} \times [0,
1])$ we obtain a well defined map
\[ \Gamma^* :  \mathcal{S}_{\del \{ 0, 1\}}(M \times D^{k+1} \times
[0, 1]) \to \mathcal{S}_{\del \{ 0, 1\}}(M \times D^{2k} \times [0,
1]),~~~[f] \mapsto [f^*(\Gamma)]. \]

We have an obvious map
\[ R: \mathcal{S}_{\del \{0, 1\}}(M \times D^{2k} \times [0, 1]) \to
\mathcal{S}_{\del \{ 0, 1\}}(M \times S^{2k-1} \times [0, 1])\]
and an obvious action
\[{\rm col} : \mathcal{S}_{\del \{0, 1\}}(M \times S^{2k-1} \times [0,1])
\times \mathcal{S}_\del(M \times D^{2k})\to \mathcal{S}_\del(M
\times D^{2k})\]
given, respectively, by restricting to the boundary and by adding a
collar.  It is straight forward to verify that there is an identity
of structures invariants
\begin{equation} \label{eqn:adding-collar}
[(h_0,\omega_0, H_0 \cup G|_{X})'] = {\rm
col}((R(\Gamma^*([G|_{X}])), [(h_0, \omega_0, H_0)']) \in
\mathcal{S}_\del(M \times D^{2k}).
\end{equation}
The following two general Lemmas then complete our proof that
$CW^{k/2}([h]) = [h']$ is well-defined.
\begin{lem} \label{lem:adding-collar-is-trivial}
Let $M$ be a closed $n$-dimensional manifold and let $l \geq 3$
be such that $n+l \geq 5$. Then the action
\[ {\rm col} \circ (R \times {\rm Id}) : \mathcal{S}_{\del \{0,
1\}}(M \times D^l \times [0, 1]) \times \mathcal{S}_\del (M \times
D^l) \to \mathcal{S}_\del(M \times D^l)\] is trivial.
\end{lem}
\begin{proof}
Let $G : X \ra M \times D^l \times [0, 1]$ be a structure
representing $[G] \in \mathcal{S}_{\del \{0, 1\}}(M \times D^l
\times [0, 1])$. Let $F : N' \ra M \times D^l$ represent $[F] \in
\mathcal{S}_\del (M \times D^l)$. Consider extending $F$ by ${\rm
col}(R(G))$, which can be conveniently denoted as $R(G) \cup F$. We
now have two structures, $F$ and $R(G) \cup F$ on $M \times D^l$ and
we need to show that they represent the same element of
$\mathcal{S}_\del(M \times D^l)$: that is, we need to find an
$h$-cobordism between them.

One way to think about $G$ is as of a homotopy between two
homeomorphisms $G_0$ and $G_1$, where $G_i \co X_i \ra M \times D^l
\times \{i\}$ are the appropriate restrictions. But we can also
think of $G$ as a rel boundary $h$-cobordism between $G_1 \cup
R(G)$ and $G_0$. This shows that $G_1 \cup R(G)$, when thought of as
an element of the structure set of $M \times D^l$ is trivial, since
$G_0$ is a homeomorphism.

Now both structures $F$ and $R(G) \cup F$ on $M \times D^l$ can be
extended by a homeomorphism to structures on $M \times S^l$, namely
$G_1 \cup R(G) \cup F$ and $G_0 \cup F$. But we can now glue $G$
considered as an $h$-cobordism as above with the trivial
$h$-cobordism $F \times \id$ to obtain an $h$-cobordism between the
two structures $G_1 \cup R(G) \cup F$ and $G_0 \cup F$ on $M \times
S^l$. This means that they represent the same element in the
structure set of $M \times S^l$. By Lemma
\ref{extension-by-homeo-is-injective} below the structures $R(G)
\cup F$ and $F$ represent the same element of $\mathcal{S}_\del(M
\times D^l)$.
\end{proof}

For the statement of the next lemma, recall that if $P$ if a closed
manifold with a decomposition $P = Q \cup C$, where $Q$ and $C$ are
codimension $0$ submanifolds. Then there is a well defined a map $E:
\sS_\del(Q) \ra \sS (P)$ given by extension with a homeomorphism.

\begin{lem} \label{extension-by-homeo-is-injective}
Let $M$ be a closed $n$-dimensional manifold and let $l \geq 3$ with
$n+l \geq 5$. Then the extension by a homeomorphism map
\[
E : \sS_\del (M \times D^l) \ra \sS (M \times S^l)
\]
is injective.
\end{lem}

\begin{proof}
We establish some notation.  Let $\pi = \pi_1(M)$, let $\Sigma^jX$ denote $j$-fold reduced suspension of a space X, let $X_+$ denote $X$ with a disjoint basepoint, let $T = M \times S^l$,  let $p : T \to M$ be the projection, let $i : M \to T$ be the obvious inclusion and let $c : T \to T/i(M) = \Sigma^lM_+$ be the collapse map so that we have a cofibration
$$M \stackrel{i}{\longrightarrow} T \stackrel{c}{\longrightarrow} \Sigma^lM_+.$$

We leave the reader to verify that $E$ fits into the following commutative diagram whose rows are fragments of the topological surgery exact sequences for $M \times D^l$ and $T = M \times S^l$.
\[
\begin{diagram}
\divide\dgARROWLENGTH by 2
\node{[\Sigma^{l+1}M_+, G/TOP]} \arrow{e,t}{\sigma_1} \arrow{s,r}{(\Sigma^1c)^*} \node{L_{n+l+1} (\pi)} \arrow{e,t}{\omega_1} \arrow{s,r}{=} \node{\sS_\del (M \times D^l)} \arrow{e,t}{\eta_1} \arrow{s,r}{E} \node{[\Sigma^lM_+, G/TOP]} \arrow{s,r}{c^*} \\
\node{[\Sigma^{1}T_+, G/TOP]} \arrow{e,t}{\sigma_2}  \node{L_{n+l+1} (\pi)} \arrow{e,t}{\omega_2} \node{\sS(T)} \arrow{e,t}{\eta_2} \node{[T, G/TOP]}
\end{diagram}
\]
Here we have identified $(M \times D^l)/(M \times S^{l-1})$ with $\Sigma^l M_+$ and $c^*$ and $(\Sigma^1 c)^*$ denote respectively precomposition with $c$ and its suspension.  Note that $G/TOP$ is an infinite loop space, $c^*$ is split by $p^*$ and so $c^*$ and $(\Sigma^1 c)^*$ are split injective.

Now let $[f_1], [f_2] \in \sS_\del(M \times D^l)$ be two structure invariants and suppose that $E([f_1]) = E([f_2])$.  As $c^*$ is injective, it follows that $[f_1]$ and $[f_2]$ have the same normal invariant and so there is an element $x \in L_{n+l+1}(\pi)$ such that $[f_1] = [f_2] + x$ where $+$ denotes the action of $L_{n+l+1}(\pi)$ on $\sS_\del(M \times D^l)$.  Hence $E([f_1]) = E([f_2]) + x$ and hence $x$ acts trivially on $E([f_1]) = E([f_2])$.

Now we use the fact that the topological surgery exact sequence is a long exact sequence of abelian groups by \cite[C.5]{Siebenmann(1977)} or \cite[Theorem 18.5]{Ranicki(1992)}. It follows that an element $y \in L_{n+l+1}(\pi)$ acts trivially on an element of $\sS(T)$ if and only if $y \in {\rm Ker}(\omega_2) = {\rm Im}(\sigma_2)$ and similarly for $\sS_\del(M \times D^l)$. We deduce that $x \in {\rm Im}(\sigma_2)$ and that it remains to show that $x \in {\rm Im}(\sigma_1)$: i.e. that  ${\rm Im}(\sigma_2) = {\rm Im}(\sigma_1)$.

By the commutativity of the above diagram ${\rm Im}(\sigma_1) \subset {\rm Im}(\sigma_2)$.  On the other hand, we see that
$$[\Sigma^1T_+, G/TOP] \cong {\rm Im}((\Sigma^1c)^*) \oplus {\rm Im}((\Sigma^1p)^*).$$
Now the geometric description of $(\Sigma^1p)^*$ is to multiply a degree one normal map to $M \times [0, 1]$ by the identity map of $S^l$ and moreover $S^l$ has trivial symmetric signature.  Using a simple modification of \cite[Proposition 8.1]{Ranicki-II-(1980)} to the relative case we conclude that $\sigma_2$ vanishes on ${\rm Im}((\Sigma^1p)^*)$. It follows that ${\rm Im}(\sigma_2) = {\rm Im}(\sigma_1)$ and hence $E$ is injective.
\end{proof}

Equation (\ref{eqn:adding-collar}) and Lemma \ref{lem:adding-collar-is-trivial} yield:

\begin{cor} \label{choice2cor} There is an equality $[(h_0, \omega_0, H_0)'] = [(h_1, \omega_0, H_0 \cup G|_{X})']$.
\end{cor}

Lemma \ref{choice1lem} and Corollary \ref{choice2cor} show that the
map $CW^{k/2}$ is well defined.

\begin{rem}
The discussion above includes, in particular, a proof of \cite[Lemma 1.1]{Hutt(1998)}: for a fixed choice of $\omega$, a choice of homotopy $H$ does not effect the equivalence class of $(h,\omega,H)' : N' \to M \times D^{2k}$ in $\mathcal{S}_\del(M \times D^{2k})$. The proof of this in the appendix of \cite{Hutt(1998)} appears to us to be incorrect: the claim in the first paragraph of Hutt's Appendix that there exists a homotopy $K$ between the homotopy $-H_2 + H_1$ and an isotopy $\mathcal{H}$ between $\psi_2$ and $\psi_1$ is not justified.
\end{rem}


\subsection{Extension by homeomorphism} \label{subsec:ext-by-homeo}
Recall that $M$ is a closed topological manifold, that $\FF = \C$ or
$\HH$ has real dimension $k = 2$ or $4$ and consider the obvious
decomposition
\[
M \times \FF P^2 = (M \times D^{2k}) \cup (M \times \FF
P^{2\bullet}) \] where $\FF P^{2\bullet} := \FF P^2 - D^{2k}$.
Associated to this decomposition extension by homeomorphism gives
the map of structure sets
\[ E : \mathcal{S}_\del(M \times D^{2k}) \to \mathcal{S}(M \times \FF P^2).\]

For the structures $h' \co N' \to M \times D^{2k}$ defined in
(\ref{def:h-prime}) above, we can realize the map $E$
as follows.  Observe that $\FF P^{2\bullet}$ is homeomorphic to the
mapping cylinder (disk bundle) of the Hopf map $\gamma \co S^{2k-1}
\ra S^k$ and recall that the restriction of the map $h'$ to $\del N'
= \del W'$ is given by the $S^{k-1}$-bundle map $\del \psi' \co \del
W' \ra M \times S^{2k-1}$ over $\del \psi \co \del W \ra M \times
S^{k}$.  Thus we can simply extend $\del \psi'$ to the associated
$D^k$-bundle map.  This amounts to extending to the associated
mapping cylinders.

To be explicit, recall that $q_W \co \del W' \ra \del W$ is the
projection in the source and define
\begin{equation} \label{defn:N-hat}
\widehat N \co = N' \cup_{\del W'} \cyl(q_W) = \cyl (p') \cup
U'_\del \cup_{\del W'} \cyl (q_W).
\end{equation}
If follows that a representative of $E[h']$ given by
\begin{equation} \label{defn:h-hat}
\widehat h := h' \cup \cyl (\del \psi',\del \psi) \co \widehat N \to
M \times \FF P^2
\end{equation}
where $\cyl (\del \psi',\del \psi)$ is the mapping cylinder
($D^{k}$-bundle) map associated to the square $\del \psi \circ q =
\gamma \circ \del \psi'$.

As with the map $h'$ we note that the map $\widehat h : \widehat N
\to M \times \F P^2$ is not a map over $h : N \to M$ since $\del
\psi$ is not a map over $h$.  On the other hand, we can replace the
structure $\widehat h$ with a structure that is over $h$.  In fact
one has the following h-decorated version of a key lemma of Hutt.
\begin{lem}{Cf. \cite[Lemma 1.4]{Hutt(1998)}.} \label{lem:phi}
There is a closed manifold $\bar N$ with a map $\bar p \co \bar N
\ra N$ and a homotopy equivalence of closed manifolds $\bar h : \bar
N \to M \times \F P^2$ covering the map $h$ which represents the
same element as $\widehat h$ in $\sS (M \times \FF P^2)$.  Indeed
there is a homotopy equivalence  $\varphi : \bar N \to N \times \F
P^2$ over $N$ such that the following diagram commutes.
\[\begin{diagram}
\node{\bar N}  \arrow{se,t}{\bar h} \arrow[2]{s,l}{\varphi} \\
\node[2]{M \times \F P^2} \\
\node{N \times \F P^2} \arrow{ne,b}{h \times {\rm Id}}
\end{diagram}
\]\end{lem}

\begin{proof}
Recall the projection map $q_{\bar W} \co \del \bar W' \ra \del \bar
W$ and also recall that the mapping cylinder $\cyl (p')$ has
boundary $\del \bar W'$. The manifold $\bar N$ is defined as
\begin{equation} \label{defn:N-bar}
\bar N := \cyl (p') \cup \cyl (q_{\bar W})
\end{equation}
and we have the obvious projection map $\bar p \co \bar N \ra N$.
Recall the homotopy equivalence $\omega : \bar W \ra N \times
D^{k+1}$ over the identity on $N$ from Lemma \ref{lem:hutt-1}.
Define $\varphi : \bar N \simeq N \times \F P^2$ by
\begin{equation} \label{defn:phi}
\varphi = \cyl (\del \omega',\id) \cup \cyl (\del \omega',\del
\omega),
\end{equation}
and $\bar h = (h \times {\rm Id}) \circ \varphi : \bar N \simeq M
\times \F P^2$.

We need to show that $\bar h$ is equivalent to $\widehat h$ and we
achieve this with a sort of Alexander trick. We first find an
$h$-cobordism $Z$ between the manifolds $\widehat N$ and $\bar N$.
To this end think of $\bar N$ as
\[
\bar N = \cyl (p') \cup (\del \bar W \times [0,1]) \cup \cyl (q_{\bar W}).
\]
Now the manifold $Z_3 := \cyl (q_{U_\del })$ yields an $h$-cobordism
between $\cyl (q_W)$ and $\cyl (q_{\bar W})$. The product $U'_\del
\times [0,1]$ can be viewed as an $h$-cobordism $Z_2$ between
$U'_\del$ and $\del \bar W' \times [0,1]$. Let $Z_1$ be the trivial
product $h$-cobordism over $\cyl (p')$. Gluing all these together
gives the desired global $h$-cobordism.

It remains to produce a homotopy equivalence from $Z$ to $M \times
\FF P^2 \times [0,1]$ which restricts to $\widehat h$ and $\bar h$
on the two ends. On $Z_1$ we just take the product of $\cyl(\del
\bar \psi',\del \bar \psi)$ with the identity. On $Z_2$ we take the
product of $H'_\del$ with the identity, but this is modified
according to the way we think of $Z_2$. On $Z_3$ we can take the map
$\cyl (H'_\del,H_\del)$. All these maps agree on the boundary and
provide the required homotopy equivalence.
\end{proof}

\begin{rem}
Notice that the only role played by $h$ in the construction of the
map $\varphi$ is to define $\chi \in [N, \G/\TOP]$.  Specifically
$\varphi$ is determined by $\del \omega$ and $\del \omega$ is
determined by $\chi \in [N ;\G/\TOP]$. Moreover the class $\chi$ is
all one needs to build the manifolds $N'$, $\widehat N$ and $\bar
N$. This observation will be useful in the last Section
\ref{sec:completion} where coboundaries for $\widehat N$ and $\bar
N$ will be constructed.
\end{rem}


\section{The algebraic theory of surgery} \label{sec:alg-sur}


We give a brief review of the algebraic theory of surgery.  In particular we review how algebraic surgery equips $\sS(M)$ with an abelian group structure by
identifying it with the algebraic structure set $\SS_{n+1} (M)$. In
more detail, the aim of this section is to define, for a closed
$n$-dimensional topological manifold, the abelian algebraic
structure group $\SS_{n+1} (M)$ and the map $s \co \sS (M) \ra
\SS_{n+1} (M)$ which turns out to be a bijection if $n \geq 5$.
Hence one can equip $\sS (M)$ with an abelian group structure via
this bijection. We will also discuss a generalization when $M$ has a
boundary. Furthermore we will discuss the functoriality of
$\SS_{n+1} (M)$ in $M$ and we will present a condition which implies
that an element in $\SS_{n+1} (M)$ is zero. All these results will
be used in subsequent sections. The principal references are
\cite{Ranicki(1992)}, \cite{Ranicki-I-(1980)},
\cite{Ranicki-II-(1980)}, and \cite{Ranicki-Weiss(2008)}.

The abelian group $\SS_{n+1} (M)$ is defined as a quadratic
$L$-group of a certain algebraic bordism category.

An {\it algebraic bordism category} $\Lambda = (\AA,\BB,\CC)$
consists of an {\it additive category with chain duality}
$(\AA,(T,e))$ and two full subcategories $\CC$, $\BB \subseteq
\BB(\AA)$ of the category of bounded chain complexes in $\AA$
satisfying certain mild assumptions.

The {\it chain duality} $(T,e)$ consists of a contravariant functor
$T \co \BB(\AA) \ra \BB(\AA)$ and a natural transformation $e \co
T^2 \ra \id$ satisfying certain conditions. It allows one to define
the tensor product $C \otimes_\AA D$ of chain complexes $C,D \in
\BB(\AA)$ such that the tensor product $C \otimes_\AA C$ becomes a
chain complex of $\ZZ[\ZZ_2]$-modules. One defines an
$n$-dimensional {\it symmetric} structure on $C \in \BB(\AA)$ to be
an $n$-cycle $\phi$ in the chain complex $W^{\%} (C) =
\Hom_{\ZZ[\ZZ_2]}(W,(C \otimes_\AA C))$, where $W$ is the standard
$\ZZ[\ZZ_2]$-resolution of $\ZZ$. An $n$-dimensional {\it quadratic}
structure on $C \in \BB(\AA)$ is defined as an $n$-cycle $\psi$ in
the chain complex $W_{\%} (C) = W \otimes_{\ZZ[\ZZ_2]} (C
\otimes_\AA C)$. An $n$-dimensional symmetric structure consists of
a collection of chains $\phi = \{ \phi_s \in (C \otimes_\AA C)_{n+s}
\; | \; s \geq 0 \}$ satisfying certain compatibility connections.
An $n$-dimensional quadratic structure consists of a collection of
chains $\psi = \{ \psi_s \in (C \otimes_\AA C)_{n-s} \; | \; s \geq
0 \}$ satisfying certain compatibility connections. There is also a
symmetrization map $(1+T) \co W_{\%} (C) \ra W^{\%} (C)$.

The pair $(C,\phi)$ is called a {\it symmetric algebraic complex},
it is $\CC$-Poincar\'e if the mapping cone of $\phi_0$, $\sC (\phi_0
\co \Sigma^n TC \ra C)$, lies in $\CC$. The pair $(C,\psi)$ is
called a {\it quadratic algebraic complex}, it is $\CC$-Poincar\'e
if $(C,(1+T)\psi)$ is $\CC$-Poincar\'e. In the above formula recall
that
\begin{equation} \label{formula:tensor-over-AA}
(C \otimes_\AA C)_n := \Hom_\AA (TC,C)_n = \Hom_\AA (\Sigma^n
TC,,C)_0.
\end{equation}
All these notions are defined in \cite[Chapters 1,3]{Ranicki(1992)}.

\subsection*{$L$-groups}
There is a well-defined notion of a cobordism of $n$-dimensional
quadratic algebraic complexes. The quadratic $L$-groups $L_n
(\Lambda)$  are the cobordism group of $n$-dimensional algebraic
complexes in $\Lambda$, that means elements are represented by those
complexes which are in $\BB \subset \BB(\AA)$ which are
$\CC$-Poincar\'e. If $\BB$ and $\CC$ are not explicitly stated, we
use the convention that $\BB = \BB(\AA)$ and $\CC = 0$. See
\cite[Definitions 1.8, 3.4]{Ranicki(1992)}

\subsection*{Example $\AA[\pi_1 (M)]$}
For any ring with involution $R$, for example for $\ZZ[\pi_1 (M)]$,
the category of finitely generated free based $R$-modules $\AA[R]$
has a chain duality given by $T(M) = \Hom_R (M,R)$. Examples
\ref{expl:sym-con}, \ref{expl:quad-con} below explain how to obtain
symmetric and quadratic algebraic complexes over the category
$\ZZ[\pi_1 (M)]$. The quadratic $L$-groups $L_n (\ZZ[\pi_1 (M)])$
defined as cobordism groups of quadratic algebraic Poincar\'e
complexes agree with the usual Wall surgery $L$-groups defined using
quadratic forms or formations \cite{Ranicki-II-(1980)}.

\subsection*{Example $\AA_\ast (K)$}
Let $K$ be a simplicial complex, or more generally a $\Delta$-set,
and let $\AA$ be an additive category with chain duality. The
category $\AA_\ast (K)$ has as its objects the so-called $K$-based
objects of $\AA$, that means objects of $\AA$ which come as direct
sums
\[
M = \sum_{\sigma \in K} M(\sigma).
\]
Morphisms are given by
\[
\Mor_{\AA_\ast (K)} (M,N) = \{ f = \sum_{\sigma, \tau \in K}
f(\tau,\sigma) \co M(\sigma) \ra N(\tau) \; | \; f(\tau,\sigma) = 0
\; \textup{unless} \; \sigma \leq \tau \}.
\]
The definition of the duality is stated in \cite[Proposition
5.1]{Ranicki(1992)}, on the objects $M \in \AA$ it is given by
\[
(TM)_r (\sigma) = T(M(\sigma))_{r+|\sigma|} \quad \textup{if} \quad
\sigma \leq \tau, |\sigma| = |\tau|-1.
\]
This formula will not be used in the present paper. What is more
important for us is the observation that an $n$-dimensional
quadratic algebraic complex $(C,\psi)$ in $\AA_\ast (K)$ includes in
particular for each $\sigma \in K$ a chain complex $C(\sigma)$ and a
duality map $\psi_0 (\sigma) \co \Sigma^n TC (\sigma) \ra C(\sigma)$
(recall (\ref{formula:tensor-over-AA})). But it contains more
information.  There are relations between these data for various
simplices and of course the components $\psi_s$ for $s > 0$. See
\cite[Definition 4.1, Proposition 5.1]{Ranicki(1992)}. Also Examples
\ref{expl:sym-con}, \ref{expl:frag-quad-con} below explain how such
complexes come from geometry.

\subsection*{Functoriality}
Let $\pi \co K \ra L$ be a $\Delta$-set map. Then we have an
additive functor
\begin{equation} \label{functor-add-cat}
\pi_\ast \co \AA_\ast (K) \ra \AA_\ast (L) \quad (\pi_\ast M)(\tau)
= \sum_{\sigma \in K, \pi(\sigma)=\tau} M(\sigma) \quad
\textup{for} \; \tau \in L
\end {equation}
which induces a functor on the chain complexes which also `commutes'
with the chain duality in a suitable sense so that one obtains a map
of the $L$-groups
\[
\pi_\ast \co L_{n+1} (\AA_\ast (K)) \ra L_{n+1} (\AA_\ast (L)).
\]
See \cite[Proposition 5.6, Example 5.8]{Ranicki(1992)}.

\subsection*{Assembly}
Slightly different functoriality is provided by the {\it assembly}
functor $A \co \ZZ_\ast (K) \ra \ZZ[\pi_1 (K)]$ defined by
\[
M \mapsto \sum_{\widetilde \sigma \in \widetilde K} = M(p(\widetilde \sigma))
\]
which also induces a map of the $L$-groups
\[
\pi_\ast \co L_n (\ZZ_\ast (K)) \cong H_n (K,\bL_\bullet) \ra L_n
(\ZZ[\pi_1 (K)]).
\]
Here the isomorphism in the source is a calculation, the symbol
$\bL_\bullet$ denotes the quadratic $\langle 0 \rangle$-connective
$L$-theory spectrum. See \cite[Chapter 9]{Ranicki(1992)} and
\cite[Chapter 13]{Ranicki(1992)} for the spectrum $\bL_\bullet$.

\subsection*{Algebraic bordism categories}
So far we have only presented examples of additive categories with
chain duality. In order to obtain an algebraic bordism category we
need to specify interesting subcategories of $\BB (\AA)$. We will
only be interested in the case $\AA =  \AA_\ast (K)$ for which we
will use three such subcategories, denoted $\DD \subset \CC \subset
\BB$. Here are the definitions:
\begin{align} \label{alg-bordism-cat}
\begin{split}
\BB & := \BB(\AA_\ast (K)), \\
\CC & := \{ C \in \BB \; | \; A(C) \simeq \ast \}, \\
\DD &: = \{ C \in \BB \; | \; C(\sigma) \simeq \ast \; \; \forall \sigma \in K \}.
\end{split}
\end{align}
This gives us three possibilities to construct interesting algebraic
bordism categories. Using a suitable notion of a functor between
algebraic bordism categories we obtain a sequence
\begin{equation} \label{loc-seq-of-alg-bor-cat}
\Lambda'' = (\AA,\CC,\DD) \lra \Lambda' = (\AA,\BB,\DD) \lra \Lambda
= (\AA,\BB,\CC)
\end{equation}
which induces a long exact sequence of groups
\begin{equation} \label{loc-seq-alg-sur}
\cdots \ra L_{n+1} (\Lambda) \ra L_n (\Lambda'') \ra L_n (\Lambda')
\ra L_n (\Lambda) \ra L_{n-1} (\Lambda'') \ra \cdots .
\end{equation}
This material is discussed in \cite[Chapter 3]{Ranicki(1992)}.

\subsection*{Connective versions}
In order to obtain groups which relate well with geometry we need to
use a connective version of the above theory. Let $q \in \ZZ$ and
let $\Lambda = (\AA,\BB,\CC)$ be an algebraic bordism category.
Define the subcategory $\BB \langle q \rangle \subset \BB$ to be the
subcategory of chain complexes in $\BB$ which are homotopy
equivalent to $q$-connected chain complexes. Further define $\CC
\langle q \rangle = \BB \langle q \rangle \cap \CC$. Then $\Lambda
\langle q \rangle = (\AA,\BB\langle q \rangle,\CC \langle q
\rangle)$ is a new algebraic bordism category. More details are
given in \cite[Chapter 15]{Ranicki(1992)}.

\begin{defn}{\cite[Chapter 17]{Ranicki(1992)}} \label{def:alg-str-set}
Let $K$ be a $\Delta$-set, $n \in \NN$ and let $\Lambda''$ be the
algebraic bordism category given by (\ref{loc-seq-of-alg-bor-cat}).
Define
\[
\SS_{n+1} (K) := L_n (\Lambda''_\ast \langle 1 \rangle).
\]
\end{defn}

\subsection*{Algebraic surgery exact sequence} Putting together the
previous statements and definitions the long exact sequence
(\ref{loc-seq-alg-sur}) becomes the algebraic surgery exact sequence
\begin{equation} \label{alg-sur-seq-concrete}
\cdots \ra L_{n+1} (\ZZ[\pi_1 (K)]) \ra \SS_{n+1} (K) \ra H_n
(K,\bL_\bullet \langle 1 \rangle) \ra L_n (\ZZ[\pi_1 (K)]) \ra
\cdots
\end{equation}
discussed in detail in \cite[Chapters 14, 15]{Ranicki(1992)}. We
will mostly work directly with the group $\SS_{n+1} (K)$, but of
course many of the properties of this group follow from the
existence of the sequence (\ref{alg-sur-seq-concrete}). For example
recall that the assignment $K \mapsto \SS_{n+1}(K)$ becomes a
covariant functor from $\Delta$-sets to abelian groups via the
functoriality described in (\ref{functor-add-cat}). So for $\pi \co
K \ra L$ a $\Delta$-set map we obtain the map
\[
\pi_\ast \co \SS_{n+1} (K) \ra \SS_{n+1} (L).
\]
In fact we obtain a map of exact sequences for $K$ and $L$ and it
follows that the functor $\SS_{n+1} (K)$ is a homotopy functor.  In
particular this means that if $\pi$ is a homotopy equivalence of
$\Delta$-sets, then $\pi_\ast$ is an isomorphism of abelian groups,
since the other two terms clearly are homotopy functors.

Next we review how the above theory relates to topology. We begin
with some remarks about topological manifolds in the above setting.
When $M$ is a closed $n$-dimensional topological manifold we can
apply Definition \ref{def:alg-str-set} only if $M$ is triangulated.
In that case it is possible to define a map $s \co \sS (M) \ra
\SS_{n+1} (M)$, whose construction we recall below and which can be
shown to be a bijection, \cite[Chapter 18]{Ranicki(1992)}. If $M$ is
not triangulated we can choose a homotopy equivalence $r \co M \ra
K$ to a finite $\Delta$-set. Such an $r$ will determine a map $s(r) \co
\sS (M) \ra \SS_{n+1} (K)$ which can be shown to be a bijection in
the same way. In both cases the bijections $s$ and $s(r)$ can be
used to give $\sS (M)$ the structure of an abelian group. It also
turns out that this group structure is independent of both the
choice of the triangulation and of the homotopy equivalence $r$.
Therefore, following Ranicki, we will abuse notation and write
$\SS_{n+1} (M)$ for $\SS_{n+1} (K)$ for any choice of a $\Delta$-set
$K$ homotopy equivalent to $M$ and $s \co \sS (M) \ra \SS_{n+1} (M)$
for $s(r)$ given by any choice of a homotopy equivalence $r \co M
\ra K$. For the record we make

\begin{defn} Let $M$ be an $n$-dimensional topological manifold
and let $r \co M \ra K$ be a homotopy equivalence to a finite
$\Delta$-set. We write
\[
\SS_{n+1} (M) = \SS_{n+1} (K).
\]
\end{defn}

Now we explain the map $s \co \sS (M) \ra \SS_{n+1} (M)$. First we
need some preparation. In the following examples we explain the
topological situations which give rise to algebraic complexes, both
symmetric and quadratic.

\begin{expl} \label{expl:sym-con}
Let $K$ be a finite $\Delta$-set with barycentric subdivision $K'$.
Consider the simplicial chain complex $C := \Delta_\ast (K')$ as an
object in $\BB(\ZZ)$. Any $n$-cycle $[K]$ in $C_n$ defines via the
symmetric construction a symmetric structure $\phi$ on $C$ over
$\ZZ$, whose component $\phi_0 \co \Sigma^n TC = C^{n-\ast} \ra
C_\ast$ corresponds to the cap product with $[K]$. If $K$ is a
closed oriented $n$-dimensional topological manifold and $[K]$ is
its fundamental class, then the resulting symmetric Poincar\'e
complex is denoted $\sigma^\ast (K)$ and is called the {\it
symmetric signature} of $K$. See \cite{Ranicki-II-(1980)}

The chain complex $C$ can also be thought of as an object in
$\BB(\ZZ_\ast (K))$ with $C(\sigma) = \Delta_\ast
(D(\sigma),\partial D(\sigma))$, the simplicial chain complex of the
dual cell relative to its boundary. Then we have $\Sigma^n TC
(\sigma) \cong \Delta^{n-|\sigma|-\ast} (D(\sigma))$.  Again by
\cite{Ranicki(1992)} for each $n$-cycle $[K] \in C_n$ there is a
refined symmetric construction.  Thus we obtain an algebraic
symmetric structure $\phi$ over $\ZZ_\ast (K)$, which in particular
contains for each $\sigma$ duality maps $\phi_0 (\sigma) \co
\Sigma^n TC (\sigma) \ra C(\sigma)$ which are cap products with
certain $(n-|\sigma|)$-dimensional classes $[K](\sigma)$. See
\cite[Example 5.5]{Ranicki(1992)}.

Now let $M$ be a topological manifold with a reference homotopy
equivalence $r \co M \ra K$ to a $\Delta$-set, which is transverse
to the dual cells of $K'$. Consider the dissection
\begin{equation} \label{dissection-of-M}
M = \bigcup_{\sigma \in K} \Big( M(\sigma) = r^{-1} (D(\sigma)).
\Big)
\end{equation}
The chain complex
\begin{equation} \label{dissection-of-chian-cplx}
C = \Sigma_{\sigma \in K} \Big( C(\sigma) = C (M(\sigma),\del
M(\sigma)) \Big)
\end{equation}
where $C (M(\sigma),\del M(\sigma))$ is the singular chain complex
of the pair $(M(\sigma),\del M(\sigma))$, yields an object in
$\BB(\ZZ_\ast (K))$. As an object in $\BB(\ZZ)$ it is weakly
homotopy equivalent to $C_\ast (M)$, the singular chain complex of
$M$. When considered as an object in $\BB(\ZZ_\ast (K))$, then,
similarly to above, there is for each $n$-cycle $[M] \in C_n$ a
refined symmetric construction, so that we obtain an algebraic
symmetric structure $\phi$ over $\ZZ_\ast (K)$, which in particular
contains for each $\sigma$ duality maps $\phi_0 (\sigma) \co
\Sigma^n TC (\sigma) \ra C(\sigma)$ which are cap products with
certain $(n-|\sigma|)$-dimensional classes $[M](\sigma)$. For more
details see \cite[Example 6.2]{Ranicki(1992)}.

\end{expl}

\begin{expl} \label{expl:quad-con}
Let $(f,b) \colon N \rightarrow M$ be a degree one normal map of
$n$-dimensional closed manifolds. Denote by $K(f)$ the algebraic
mapping cone of the Umkehr map of chain complexes
\[
f^! \colon C_\ast\widetilde{M} \simeq C^{n - \ast} \widetilde{M}
\xra{f^{n -\ast}} C^{n - \ast} \widetilde{N} \simeq C_\ast
\widetilde{N}.
\]
Then $C_*\widetilde M$ comes with a structure of an $n$-dimensional
symmetric algebraic Poincar\'{e} complex over $\ZZ[\pi_1 (M)]$. This
projects to a structure of an $n$-dimensional symmetric algebraic
Poincar\'{e} complex on $K(f)$. In \cite{Ranicki-II-(1980)} this is
refined to an $n$-dimensional quadratic algebraic Poincar\'{e}
complex on $(K(f),\psi(f))$.
\end{expl}

\begin{expl} \label{expl:frag-quad-con}
Let $(f,b) \colon N \rightarrow M$ be a degree one normal map of
closed $n$-dimensional manifolds and let $r\colon M \rightarrow K$
be a map to a $\Delta$-set $K$ such that both $rf$ and $r$ are
transverse to the dual cells of $K$. There are $K$-dissections $N
\cong \cup N(\sigma)$ and $M \cong \cup M(\sigma)$, so that $C_*N$
and $C_*M$ can be regarded as objects in $\BB(\ZZ_\ast(K))$. There
are preferred structures of $n$--dimensional symmetric algebraic
complexes on $C_*N$ and $C_*M$, as objects of $\BB(\ZZ_\ast(K))$
coming from the fundamental classes. By analogy with
Example~\ref{expl:quad-con}, there is an algebraic Umkehr map
\[
f^!\co C_*M \lra C_*N
\]
in $\BB(\ZZ_\ast(K))$ with mapping cone $K(f)$, say. The resulting
structure of and $n$--dimensional symmetric algebraic complex on
$K(f)$, as an object of $\BB(\ZZ_\ast(K))$, has a preferred
refinement to a quadratic structure $\psi(f)$. The chain complex
$K(f)(\sigma)$ for a $\sigma \in K$ can be identified with the
mapping cone of an algebraic Umkehr map
\[
C_*(M(\sigma),\partial M(\sigma)) \lra C_*(N(\sigma),\partial
N(\sigma)).
\]
See \cite[Examples 9.13, 9.14]{Ranicki(1992)} for details. Under
assembly, this construction coincides with that in
Example~\ref{expl:quad-con}.
\end{expl}

\begin{expl} \label{expl:frag-quad-con-over-htpy-equivalence}
Note that $(K(f),\psi(f))$ is $\DD$-Poincar\'e. When in addition $f$
is a homotopy equivalence, then $K(f)$ is contractible after
assembly. Furthermore the required connectivity assumptions are
fulfilled so that the pair $(K(f),\psi(f))$ represents an element in
$\SS_{n+1} (M)$.
\end{expl}

\begin{defn}{\cite[Proposition 18.3]{Ranicki(1992)}}
The map
\[
s \co \sS (M) \ra \SS_{n+1} (M), ~~~~~~[f: N \to M] \mapsto [K(f), \psi(f)]
\]
is defined by the construction described in Examples
\ref{expl:frag-quad-con},
\ref{expl:frag-quad-con-over-htpy-equivalence}.
\end{defn}

\begin{expl}{\cite[Proposition 18.3]{Ranicki(1992)}}
In case we deal with a manifold with boundary $(Y,\partial Y)$, the
constructions in Examples \ref{expl:frag-quad-con},
\ref{expl:frag-quad-con-over-htpy-equivalence} yield a map from
$\sS_\partial (Y)$ to $\SS_{n+1} (Y)$. When $Y = M \times D^k$, then
thanks to the homotopy invariance of $\SS_{\ast} (-)$ we obtain a
map
\[
s \co \sS_\partial (M \times D^k) \ra \sS_{n+1+k} (M).
\]
\end{expl}

\begin{thm}{\cite[Theorem 18.5]{Ranicki(1992)}} \label{thm:identifcation-geo-alg-str-sets}
For a closed manifold $M$ with dimension $n \geq 5$ we have
\[
s \co \sS (M) \xra{\equiv} \SS_{n+1} (M), \quad \quad s \co \sS_\partial (M \times D^k) \xra{\cong} \sS_{n+1+k} (M),
\]
where $\equiv$ means a bijection and $\cong$ means an isomorphism of
abelian groups.
\end{thm}

\begin{rem} \label{rem:composition-formula}
Suppose that $h \co M \ra M'$ is a homotopy equivalence of $n$-dimensional closed manifolds. Then the assignment $[f] \mapsto [h \circ f]$ provides us with a map of sets $\sS (M) \ra \sS (M')$. On the other hand $h$ induces a homomorphism $h_\ast \co \SS_{n+1} (M) \ra \SS_{n+1} (M')$. We note that in general $s ([h \circ f]) \neq h_\ast s ([f])$, but there is the composition formula $s ([h \circ f]) = h_\ast s ([f]) + [h]$, see \cite{Ranicki(2009)}. This will indeed be used later in the proof of Proposition \ref{prop:sieb}.
\end{rem}

\begin{expl} \label{functoriality}
Let $h \co P \ra N$ be a homotopy equivalence of closed
$n$-dimensional manifolds representing an element in $\sS (N)$.
Given a homotopy equivalence $r : N \to K$ to a finite $\Delta$-set
$K$ we have described $s ([h]) \in \SS_{n+1} (N)$ in Examples
\ref{expl:frag-quad-con},
\ref{expl:frag-quad-con-over-htpy-equivalence}. Let $s \co M \ra L$
be a homotopy equivalence from another closed $n$-dimensional
manifold to a finite $\Delta$-set and let $f \co N \ra M$ be a map
covering via the reference maps a $\Delta$-set map $\pi \co K \ra
L$.  Consider $\pi_\ast (s([h])) \in \SS_{n+1} (M)$.  From the
description of the functoriality in (\ref{functoriality}) we see
that for each $\tau \in L$, $\pi_\ast (s([h]))(\tau )$ has the underlying chain complex the algebraic mapping cone of the map
\[ \big(\bigcup_{\pi(\sigma) = \tau} P(\sigma),\partial \big) \lra
\big( \bigcup_{\pi(\sigma) = \tau} N(\sigma),\partial\big).
\]
See \cite[Example 5.8]{Ranicki(1992)}.
\end{expl}

\begin{expl} \label{functoriality-projection}
We will need a special case of the above example when $\pi$ is the
projection map $\pi_1 \co K \otimes L \ra K$. Here $K \otimes L$ is
the geometric product of $\Delta$-sets, see
\cite{Rourke-Sanderson(1971)} or \cite[Chapter 11]{Ranicki(1992)}. A
$p$-simplex of $K \otimes L$ is a triple
\[
(\sigma,\tau,\lambda) \quad \textup{where} \quad \sigma \in
K^{(m)}, \tau \in L^{(n)}, \lambda \in (\Delta^m \otimes
\Delta^n)^{(p)}
\]
with $\Delta^m \otimes \Delta^n$ the product of ordered simplicial
complexes. There is a homeomorphism $|K \otimes L| = |K| \times |L|$
and there is a projection map $\pi_1 \co K \otimes L \ra L$ which is
a $\Delta$-set map (the explicit formula is easy but a little
complicated and we do not need it). We have
\[
\bigcup_{\tau,\lambda} D(\sigma,\tau,\lambda) = D(\sigma) \times L.
\]
Let $M$ and $N$ be two closed topological manifolds with reference
homotopy equivalences to $\Delta$-sets $r \co M \ra K$ and $r' \co N
\ra L$ transverse to the dual cells. Then the product map $r \times
r' \co M \times N \ra |K| \times |L|$ is transverse to the dual
cells of the geometric product of the $\Delta$-sets $K \otimes L$.
Let $h \co P \ra M \times N$ be a simple homotopy equivalence
representing an element in $\sS (M \times N)$ which is transverse to
the dissection of $M \times N$ induced by $r \times r'$. We have
$s([h]) \in \SS_{m+n+1} (M \times N)$ and this element is
represented by an algebraic Poincar\'e complex over $\ZZ_\ast (K
\otimes L)$ whose value at $(\sigma,\tau,\lambda) \in K \otimes L$
has its underlying chain complex the mapping cone of the Umkehr map
of the degree one normal map
\begin{equation} \label{eqn:h-before-partial-assembly}
h(\sigma,\tau,\lambda) \co P (\sigma,\tau,\lambda) = h^{-1} (M
\times N) (\sigma,\tau,\lambda) \ra (M \times N)
(\sigma,\tau,\lambda).
\end{equation}
Consider the projection map $p \co M \times N \ra M$. We have
$p_\ast (s ([h])) \in \SS_{m+n+1} (M)$. It follows from the above
discussion that this element is represented by an algebraic
Poincar\'e complex in $\AA_\ast (K)$ whose value at $\sigma \in K$
has its underlying chain complex the mapping cone of the Umkehr map
of the degree one normal map
\begin{equation} \label{eqn:partial-assembly}
h(\sigma) \co P (\sigma) = h^{-1} (M(\sigma) \times N) \ra M
(\sigma) \times N.
\end{equation}
It may happen that such $h$ represents a non-trivial element in $\sS
(M \times N)$ and hence $s([h])$ is a non-zero element in
$\SS_{m+n+1} (M \times N)$ and at the same time for each $\sigma \in
K$ the map \ref{eqn:partial-assembly} is a simple homotopy
equivalence. Then the underlying chain complex for each $\sigma$
is contractible and hence the projection $p_\ast(s([h])) = 0 \in
\SS_{m+n+1} (M)$ by the Proposition \ref{prop:criterion-zero} below.
Such a situation will indeed occur in the next section.
\end{expl}

\begin{prop} \label{prop:criterion-zero}
Let $(C,\psi)$ represent an element in $\SS_{n+1} (M)$. Suppose in
addition that $C(\sigma) \simeq \ast$ for each $\sigma \in K$. Then
\[
[(C,\psi)] = 0 \in \SS_{n+1} (M).
\]
\end{prop}

\begin{proof}
The homotopy equivalences for each $\sigma$ assemble to a
null-bordism of chain complexes in the algebraic bordism category
$\Lambda'' \langle 1 \rangle$ with $\Lambda'' = (\ZZ_\ast
(K),\CC,\DD)$ as in (\ref{alg-bordism-cat}) which defines $\SS_{n+1}
(M)$.
\end{proof}

One of the important and useful features of the algebraic theory of
surgery is a particularly easy description of periodicity. Indeed,
in case one works with the $0$-connective version of the algebraic
structure set, denoted by ${\bar \SS}_{n+1} (M)$, one obtains the
$4$-periodicity given by the so-called skew-double-suspension:
\begin{equation} \label{eqn:4-periodocity-S-bar}
{\bar S}^2 : {\bar \SS}_{n+1} (M) \ra {\bar \SS}_{n+5} (M).
\end{equation}
If one works with the $1$-connective version and for a positive
integer $k$ sets $S^{2k} := (S^2)^k$ one obtains in general an exact
sequence
\begin{equation} \label{eqn:4-periodocity-S}
0 \ra \SS_{n+1} (M) \xra{S^{2k}}  \SS_{n+2k+1} (M) \ra H_n (M,L_0
(\ZZ)) \ra \cdots
\end{equation}
where in fact $H_n (M,L_0 (\ZZ)) \cong H_n (M,\ZZ)$: see \cite[Remark 25.4]{Ranicki(1992)}.

This near-periodicity can also be defined using products in
$L$-theory. Recall the symmetric signature $\sigma^\ast (M)$ of an
$n$-dimensional Poincar\'e complex $M$, a symmetric algebraic
Poincar\'e complex over $\ZZ$. The products in algebraic surgery
\cite[Appendix B]{Ranicki(1992)} give for an $n$-dimensional
quadratic algebraic $\CC$-Poincar\'e complex $(C,\psi)$ representing
an element in $\SS_{n+1} (M)$ a new $(n+4)$-dimensional quadratic
algebraic $\CC$-Poincar\'e complex $(C,\psi) \otimes \sigma^\ast
(\CC P^2)$ representing an element in $\SS_{n+5} (M)$. This produces
a map which coincides with the double skew-suspension.  In geometry
this map corresponds to taking a product with the identity on $\CC
P^2$ and projecting algebraically.

More generally one has the following identity of injective
homomorphisms
\begin{equation} \label{eqn:S=sigma}
(\otimes{\sigma^\ast (\FF P^2)} = S^{k} ) \co \SS_{n+1} (M) \ra
\SS_{n+2k+1} (M)
\end{equation}
where $S^{k}$ is $2k$-skew-suspension map and $\otimes{ \sigma^\ast
(\FF P^2)}$ is the homomorphism defined by taking the product with
the symmetric signature of $\CC P^2$ or $\HH P^2$ for $k = 2$ or $4$
respectively.


\section{Siebenmann periodicity} \label{sec:sieb-per}


Recall that $M$ is a closed topological manifold of dimension $n
\geq 5$ and Theorem \ref{thm:A} which states that the
Cappell-Weinberger map $CW^2 : \sS(M) \to \sS_\del(M \times D^8)$ is
an injective homomorphism with cokernel a subgroup of $\Z$.  Theorem
\ref{thm:A} is a direct consequence of the Proposition
\ref{prop:sieb} below.  The exactness part follows from the
exactness of $S^{4}$ in (\ref{eqn:4-periodocity-S}) and the identity $
\otimes \sigma^\ast(\HH P^2) = S^{4}$ of (\ref{eqn:S=sigma}).

\begin{prop} \label{prop:sieb}
For $\F = \HH$, hence $k = 4$, the following diagram commutes.
\[
\xymatrix{
& \sS (M) \ar[rr]^{CW^{k/2}} \ar[d]_{s} & & \sS_\partial (M \times D^{2k}) \ar[d]^{s} & & \\
0 \ar[r] & \SS_{n+1} (M) \ar[rr]^{\otimes \sigma^\ast (\F P^2)} & &
\SS_{n+2k+1} (M) \ar[r] & H_n(M; L_0(\Z)) }
\]
\end{prop}

\begin{proof}
Recall that besides the map $CW^{k/2}$ from $\sS (M)$ to $\sS_\partial
(M \times D^{2k})$ we have also discussed the extension by a
homeomorphism map $E$ which brings us further to $\sS ( M \times \FF
P^2)$. This map will be helpful in the proof, in fact the situation
can be described by the following diagram:
\[
\xymatrix{ \mathcal{S}^{}(M) \ar[rrr]^{CW^{k/2}} \ar@{.>}[dr]^{\times \F
P^{2}} \ar[ddd]_{s} & & & \mathcal{S}^{}_\partial (M \times D^{2k})
\ar[dll]_{E} \ar[ddd]^{s} \\ & \mathcal{S}^{}(M \times \F P^2)
\ar[dr]^{s_{M \times \F P^2}} \\ & & \SS_{n+2k+1}(M \times \F P^2)
\ar[dr]^{p_\ast}\\ \SS_{n+1}(M) \ar[rrr]_{\otimes \sigma^*(\F P^2)}
& & & \SS_{n+2k+1}(M) }
\]
The discussion at the end of the last section shows that the lower
triangle commutes. The triangle on the right commutes as well. We
warn the reader that we do not claim that the upper triangle
commutes, in fact it does not (that's why the arrow is dotted).\footnote{Just before his Lemma 1.4, \cite[p.\,296  ]{Hutt(1998)}, Hutt incorrectly stated that the upper triangle commutes.  However, it seems from the rest of the paper that this may well have been a typographical error.}
Nevertheless we will show that the outer square commutes. For this
we first recall Lemma \ref{lem:phi} which says that $E (CW^2 ([h]))
\simeq [(h \times \id) \circ \varphi]$ where $\varphi \co \bar N \ra
N \times \FF P^2$ is a certain homotopy equivalence over the
identity of $N$. The proof of the proposition boils down to the
following
\begin{lem} \label{lem:phi-alg-sur}
For $\FF = \HH$, hence $k = 4$, there is an equality
\[
p_\ast \big( s ([\varphi]) \big) = 0 \in \SS_{n+2k+1} (N)
\]
where $p_\ast \co \SS_{n+2k+1} (N \times \F P^2) \ra \SS_{n+2k+1}
(N)$ is the homomorphism induced by the projection $p \co N \times
\F P^2 \ra N$.
\end{lem}
We finish the proof of the Proposition \ref{prop:sieb} and then
prove Lemma \ref{lem:phi-alg-sur}. We have the following equalities:
\begin{align}
\begin{split}
s (CW^{k/2} (h)) & = \pr_1 s (E \circ CW^{k/2} ([h]))  \\
& = \pr_1 s ([(h \times \id) \circ \varphi])   \\
& = \pr_1 ((h \times \id)_\ast s ([\varphi]) + s ([h \times \id]))  \\
& = h_\ast \pr_1 s ([\varphi]) + s ([h]) \otimes \sigma^\ast (\FF P^2)  \\
& = s ([h]) \otimes \sigma^\ast (\FF P^2).
\end{split}
\end{align}
The first equality follows from the definitions and the
functoriality of $\SS_?(-)$, the second from Lemma \ref{lem:phi},
the third from the composition formula of \cite{Ranicki(2009)}, the
fourth again from the functoriality of $\SS_?(-)$ and the fifth from
Lemma \ref{lem:phi-alg-sur}.
\end{proof}

\begin{proof}[Proof of Lemma \ref{lem:phi-alg-sur}]

\

Recall the steps that lead from a homotopy equivalence $h : N \to M$
to the homotopy equivalence $\varphi \co \bar N \ra N \times \FF
P^2$ defined by (\ref{defn:phi}):
\begin{enumerate}
\item Start with a map $\chi \co N \ra \G/\TOP$ (which was chosen so that $(h \times \id)^\ast\chi = -[h \times \id] \in [M, \G/\TOP]$).
\item Construct the homotopy equivalence of pairs $\omega \co (\bar W,\del) \ra (N
\times D^{k+1},\del)$ over the identity from the MCN $p \co (\bar
W,\del) \ra N$, (in Proposition
\ref{prop:constructing-neighborhoods}).
\item Consider the restriction, $\del \omega \co \del \bar W \ra N
\times S^{k}$ which is a homotopy equivalence over the identity.
\item Construct the homotopy equivalence $\del \omega' \co \del \bar W' \ra N
\times S^{2k-1}$ over the identity as the pullback of $\gamma_N \co
N \times S^{2k-1} \ra N \times S^k$ along $\del \omega$ (recall
(\ref{defn:del-W-prime})). This yields projection maps $q_{\bar W}
\co \del \bar W' \ra \del \bar W$ and $p' \co \del \bar W' \ra N$.
\item Define $\bar N = \cyl (p') \cup \cyl (q_{\bar W})$, (Formula
(\ref{defn:N-bar})).
\item Define $\varphi = \cyl (\del \omega',\id) \cup \cyl (\del
\omega',\del \omega)$, (Formula (\ref{defn:phi})).
\end{enumerate}

Next recall $r \co M \ra K$ a homotopy equivalence from $M$ to a
$\Delta$-set $K$ which is transverse to the dual cells of $K$ so
that we have a dissection
\[
M = \cup_{\sigma \in K} M(\sigma)
\]
with $M(\sigma) = r^{-1} (D(\sigma,K))$ a submanifold with boundary
of dimension $(n - |\sigma|)$. Further assume that $h \co N \ra M$
is transverse to $M (\sigma)$ for each $\sigma$ so that $N (\sigma)
= h^{-1} (M(\sigma))$ is a submanifold with boundary of dimension
$(n-|\sigma|)$ and $h (\sigma) \co N(\sigma) \ra M(\sigma)$ is a
degree one normal map. We obtain a dissection
\begin{equation} \label{N-dissection}
N = \cup_{\sigma \in K} N(\sigma).
\end{equation}

\nin \textbf{Geometry.} We show that by a small homotopy it is
possible to change $\chi \co N \ra \G/\TOP$ so that $\bar N$ possess
a dissection indexed by simplices $\sigma \in K$, the map $p$
respects the dissections of $\bar N$ and $N$:
\begin{equation} \label{dissection-of-p-hat}
\bar p = \cup \bar p(\sigma) \co \bar N  = \cup \bar N (\sigma) \ra
N = \cup N(\sigma)
\end{equation}
and the homotopy equivalence $\varphi \co \bar N \ra N \times \FF
P^2$ also respects these dissections. Furthermore for each $\sigma
\in K$
\begin{equation} \label{dissection-of-phi}
\varphi (\sigma) \co  \bar N (\sigma) \ra N (\sigma) \times \FF P^2
\end{equation}
is a homotopy equivalence.

To this end modify the map $\chi$ by a small homotopy so that when
restricted to the collar of each manifold with boundary
$(N(\sigma),\del N(\sigma))$ it is the product map with the identity
in the collar direction. Hence we have
\[
\chi = \cup \chi (\sigma) \co N = \cup N(\sigma) \ra \G/\TOP.
\]

To achieve this we follow steps (1)-(6) above but in the relative setting. Note that this requires Proposition \ref{prop:constructing-neighborhoods-relative} instead of Proposition \ref{prop:constructing-neighborhoods} in step (2). We will proceed inductively starting with simplices $\sigma$ of the top dimension, since then  $N(\sigma)$ has dimension $0$. Over such $N(\sigma)$ the steps (1) to (6) are trivial. To make the inductive step note that as $k = 4$ the dimension restrictions Proposition \ref{prop:constructing-neighborhoods-relative} are fulfilled, since by the inductive assumption the dimension of $N(\sigma)$ is $\geq 1$. The steps (3) to (6) have straightforward generalizations to the relative case.

The manifold $\bar N$ is the union of all the manifolds $\widehat
N(\sigma)$ just constructed and the projection map $\bar p \co \bar
N \ra N$ is the union of the corresponding projections maps $\bar
p(\sigma)$. Similarly the homotopy equivalence $\varphi$ is the
union of all the homotopy equivalences $\varphi(\sigma)$.

\

\nin \textbf{Algebraic surgery.} The homotopy equivalence $\varphi$
represents an element in the structure set $\sS (N \times \FF P^2)$.
The map $s \co \sS (N \times \FF P^2) \ra \SS_{n+2k+1} (N \times \FF
P^2)$ was described in Section \ref{sec:alg-sur}. To use it we need
to choose a $\Delta$-set homotopy equivalent to $N \times \FF P^2$.
Since $\FF P^2$ is a triangulable manifold we can choose a
triangulation and denote the underlying $\Delta$-set by $L$, the
reference map will be denoted $r' \co \FF P^2 \ra L$. Then we can
pick as our choice the geometric product $K \otimes L$, whose
geometric realization we identify with the product $|K| \times |L|$,
and the reference map $\bar r \co = (h \circ r) \times r' \co N
\times \FF P^2 \ra |K| \times |L|$ is automatically transverse to
the dual cells of $K \otimes L$ which we consider as the underlying
space of the geometric product of $\Delta$-sets described in Section
\ref{sec:alg-sur}.

Note that each dual cell of $K \otimes L$ is a subspace of $D
(\sigma) \times L$ for suitable $\sigma \in K$. Hence also
\[
(N \times \FF P^2) (\sigma,\tau,\lambda) \subset N (\sigma) \times
\FF P^2
\]
for each $\tau, \lambda$. In fact
\[
\bigcup_{\tau,\lambda} (N \times \FF P^2) (\sigma,\tau,\lambda)
\subset N (\sigma) \times \FF P^2.
\]
To determine a representative of $s([\varphi])$ in  $\SS_{n+2k+1} (N
\times \FF P^2)$, where we work over the category $\ZZ_\ast (K
\otimes L)$, the map $\varphi$ needs to be made transverse to the
submanifolds $(N \times \FF P^2) (\sigma,\tau,\lambda)$ for each
$(\sigma,\tau,\lambda) \in K \otimes L$. This can be done by a small
homotopy which does not spoil the property that $\varphi$ respects
the dissections of $\bar N$ and $N \times \FF P^2$ over $K$ and that
each $\varphi (\sigma)$ is a homotopy equivalence. To achieve this
we can again proceed inductively starting from the simplices $\sigma
\in K$ of the top dimension. We change each $\varphi (\sigma)$ by a
small homotopy to make it transverse to $(N \times \FF
P^2)(\sigma,\tau,\lambda)$ for all choices of $\tau$ and $\lambda$,
which if course does not spoil the fact that it is a homotopy
equivalence. Hence the new $\varphi$ is the union of the new
$\varphi (\sigma)$ and hence is transverse to $(N \times \FF P^2)
(\sigma,\tau,\lambda)$ for each $(\sigma,\tau,\lambda) \in K \otimes
L$.

Now we find ourselves in the situation described in Example
\ref{functoriality-projection}. We have the homotopy equivalence
$\varphi \co \bar N \ra N \times \FF P^2$ whose image $s([\varphi])
\in \SS_{n+2k+1} (N \times \FF P^2)$ is represented by a quadratic
chain complex over the category $\ZZ_\ast (K \otimes L)$ whose value
at each $(\sigma,\tau,\lambda) \in K \otimes L$ has its underlying
chain complex the mapping cone of the Umkehr map of the degree one
normal map
\begin{equation} \label{eqn:phi-before-partial-assembly}
\varphi (\sigma,\tau,\lambda) \co \bar N (\sigma,\tau,\lambda) =
\varphi^{-1} (N \times \FF P^2) (\sigma,\tau,\lambda) \ra (N \times
\FF P^2) (\sigma,\tau,\lambda).
\end{equation}
This may very well be a representative of a non-zero element in
$\SS(N \times \FF P^2)$.

But we are really interested in the projection $p_\ast \big( s
([\varphi]) \big) \in \SS_{n+2k+1} (N)$. By Example
\ref{functoriality-projection} this is represented by a quadratic
chain complex over the category $\ZZ_\ast (K)$ whose value at each
$\sigma \in K$ has its underlying chain complex the mapping cone of
the Umkehr map of the degree one normal map
\begin{equation} \label{eqn:phi-partial-assembly}
\varphi (\sigma) \co \bar N(\sigma) \ra N (\sigma) \times \FF P^2.
\end{equation}
But $\varphi (\sigma)$ is a homotopy equivalence for each $\sigma$
and so the resulting chain complex over each $\sigma$ is
contractible.  Thus by Proposition \ref{prop:criterion-zero}
$p_*(s([\varphi])) = 0$ as required.
\end{proof}

\begin{rem}
The above proof shows why we chose $k = 4$. If $k = 2$, then the
dimension restrictions of Proposition
\ref{prop:constructing-neighborhoods-relative} are not
satisfied.\footnote{It is possible that the dimension restrictions
in the relevant proposition can be relaxed. This would require
careful analysis of all the tools used in the proofs. This might be
an interesting problem but lies beyond the scope of this paper.}
\end{rem}


\section{The bordism groups $\Omega^{\STOP}_{2d-1} (G/\TOP \times BG$)} \label{sec:bordism-groups}

Let $X$ be a space and let $\Omega_n^{\STOP}(X)$ denote the $n$th
oriented topological bordism group of $X$.  Recall that $G$ is a
finite group and $BG$ is its classifying space.  The purpose of this
section is to prove the following
\begin{lem} \label{lem:bord}
For all $d \geq 1$, $\Omega_{2d-1}^{\STOP}(\G/\TOP \times BG) \tensor
\Q = 0$.
\end{lem}
\begin{proof}
The functor $X \to \Omega_*^{\STOP}(X) \tensor \Q$ is a generalised
homology theory with coefficients $\Omega_*^{\STOP} \tensor \Q$. By
Theorem \ref{thm:bord} for the trivial group
$\Omega_{2d-1}^{\STOP}({\rm pt}) \tensor \Q  = 0$. Applying the
Atiyah-Hirzebruch spectral sequence to compute $\Omega_*^{\STOP}(X)$
\[
\bigoplus_{p+q = *}H_p(X; \Omega_q^{\STOP} \otimes \Q)
\Longrightarrow \Omega_{*}^{\STOP}(X) \otimes \Q,
\]
we see that if $H_{2d-1}(X; \Q) \cong 0$ for all $d$ then
$\Omega_{2d-1}^{\STOP}(X) \tensor \Q = 0$ for all $d$.

Now applying \cite[Remark 4.36]{Madsen-Milgram(1979)} the space $\G/\TOP$ is rationally a product of Eilenberg-MacLane spaces $K(\QQ,4i)$ for $i \geq 1$ and so $H_{2d-1}(\G/\TOP; \Q) = 0$ for all $d$. As $G$ is a finite group $H_*(BG; \Q) = 0$ for all $*> 0$ and so by the Kunneth Theorem we see that $H_{2d-1}(\G/\TOP \times BG; \Q) = 0$ for all $d$. Thus we conclude that $\Omega_{2d-1}^{\STOP}(\G/\TOP \times BG) \tensor \Q = 0$ for all $d$.
\end{proof}

\section{Completion of the proof of Theorem \ref{thm:main}} \label{sec:completion}
In this section we prove Theorem \ref{thm:B} which completes the
proof of Theorem \ref{thm:main}. Recall the definition of the maps
$\wrho$, $\wrho_\partial$ and the $CW^{k/2}$-map and that $k = 2$ or
$4$. In addition recall the map defined by extension by a
homeomorphism $E \co \sS_\partial (M \times D^{2k}) \ra \sS (M
\times \F P^2)$. Let $h \co N \ra M$ represent an element in $\sS
(M)$, let $h' \co N' \ra M \times D^{2k}$ represent $CW^{k/2}
([h])$. Recall from Lemma \ref{lem:phi} that $E([h'])$ can be
represented by two homotopy equivalences, namely either by $\widehat
h \co \widehat N \ra M \times \F P^2$ (see (\ref{defn:h-hat})) or by
$\bar h \co \bar N \ra M \times \F P^2$ (see just below (\ref{defn:phi})).

\begin{lem} \label{lem-1}
There are identities
\begin{enumerate}
\item $\rho (\bar N) = \wrho_\partial ([h']) + \rho (M)$,
\item $\rho (\bar N) = \rho (N)$.
\end{enumerate}
\end{lem}
\begin{proof}[Proof of Theorem \ref{thm:B}]
Combining Lemma \ref{lem-1} (1) and (2) we have:
\[
\wrho_\partial ([h']) = \rho (N) - \rho (M) = \wrho ([h]).
\]
\end{proof}

\begin{proof}[Proof of Lemma \ref{lem-1}]
(1) The argument is similar to the proof of Proposition
\ref{rho-add-for-bdry}.  From the rel. boundary structure $h' : N'
\to M \times D^{k}$ we form the closed manifold $M(h') := N'
\cup_{h'}(-M \times D^{2k})$ and by Definition
\ref{defn:reduced-rho-del}, $\wrho_\del([h']) = \rho(M(h'))$. Recall
the operation $\#_M$ in \ref{defn:ctd-sum-along-M} and observe that
\[
M \times \FF P^2 = (M \times \FF P^{2\bullet}) \cup (M \times D^{2k}).
\]
If follows that we can form the closed manifold
\begin{equation} \label{last-ctd-sum}
M(h') \#_M (M \times \FF P^2)
\end{equation}
just as in Definition \ref{defn:ctd-sum-along-M}. By a similar
reasoning as in the proof of Proposition \ref{rho-add-for-bdry} we
obtain that the $\rho$-invariant of the manifold in
(\ref{last-ctd-sum}) is the sum of $\rho$-invariants $\rho (M(h'))$
and $\rho (M \times \FF P^2)$. Concerning the signature defect term in this setting it is enough to observe that the modules $A$, $B$ and $C$ appearing in Wall's definition of the signature defect term \cite[Theorem p.217]{Wall(1969)} are equal to the module
$K$ appearing in the proof of Proposition \ref{rho-add-for-bdry}.  The first two because they are the kernels 
\[ 
A = B = K = {\rm Ker} \left( H_{d-1}(M \times S^{2k-1}) \to H_{d-1}(M \times D^{2k}) \right), 
\]
the third one because we have
\[
K = C = {\rm Ker} \left( H_{d-1}(M \times S^{2k-1}) \to H_{d-1}(M \times \FF P^{2\bullet}) \right). 
\]
From the construction of $\widehat N$ in (\ref{defn:N-hat}) we also see that there is a homeomorphism
\[
\widehat N \cong M(h') \#_M (M \times \FF P^2 ).
\]
The statement now follows since $\rho(M \times \FF P^2) = \rho(M)$
and $\rho (\widehat N) = \rho (\bar N)$, which we have from the
$h$-cobordism invariance of the $\rho$-invariant.

(2) We are given $[h : N \to M] \in \mathcal{S}(M)$ and $[ \bar h
\co  \bar N \to M \times \F P^2] \in \mathcal{S}(M \times \F P^2)$
which represents $E \circ CW^{k/2}([h])$ and we wish to prove that
$\rho(\bar N) = \rho(N)$ where the reference map for $\bar N$ is
$\lambda(N) \circ \bar p$ and $\bar p : \bar N \to N$ is the map
constructed in Section \ref{subsec:ext-by-homeo}.

Recall from Definition \ref{defn:rho-1} the definition of the
$\rho$-invariant of a $(2d-1)$-dimensional manifold $N$ equipped
with a map $\lambda (N) \co N \ra BG$ inducing $\lambda (N)_\ast \co
\pi_1  (N) \ra G$, with $G$ a finite group.  Because
$\Omega^{\STOP}_{2d-1} (BG) \otimes \QQ = 0$ there is a coboundary
for $\sqcup_{i=1}^r N$ over $\lambda (N)$ for some $r \geq 1$.  That
is, there is a manifold $P$ with boundary $\del P = \sqcup_{i=1}^r
N$, and with a map $\lambda (P) \co P \ra BG$ extending $\sqcup^r
\lambda (N)$. The formula is
\begin{equation} \label{rho-of-N}
\rho(N) := (1/r)\Gsign (P).
\end{equation}

To show the desired statement it is enough to find a coboundary, say
$\bar P$, for $\sqcup_{i=1}^r \bar N$ over $\lambda (\bar N) =
(\lambda (N) \circ \bar p) \co \bar N \ra N \ra BG$, such that $\bar
P \simeq P \times \FF P^2$. Then by the multiplicativity of
$G$-signature we would obtain
\begin{equation} \label{rho-is-mult}
\rho(\bar N) = (1/r)\Gsign (\bar P) = (1/r)\Gsign (P) \cdot
\sign(\FF P^2) = \rho(N).
\end{equation}

Recall from Section \ref{sec:cw-map} that the closed manifold $\bar
N$ along with a homotopy equivalence $\varphi \co \bar N \ra N
\times \FF P^2$ was constructed from a map $\chi (N) \co N \ra
\G/\TOP$.  Now, by Lemma \ref{lem:bord} we have
$\Omega^{\STOP}_{2d-1} (BG \times \G/\TOP) \otimes \QQ = 0$. This
implies that there exists a manifold $P$ with boundary $\del P =
\sqcup_{i=1}^r N$ and a map
\[
\kappa (P) \co P \ra BG \times \G/\TOP
\]
such that
\[ \pr_{\G/\TOP} \circ (\kappa (P)|_{\del P})  =  \sqcup_{i=1}^r \chi (N) \co (\sqcup_{i=1}^r N) \ra \G/\TOP\]
and
\[ \pr_{BG} \circ (\kappa (P)|_{\del P}) = \sqcup_{i=1}^r \lambda (N)  \co (\sqcup_{i=1}^r N ) \ra BG.\]
Here $\pr_{\G/\TOP}$ and $\pr_{BG}$ are the obvious projections.  We
have used the same letter $P$ as above because such a $P$ can be
used as a coboundary of $\sqcup_{i=1}^r N$ in \ref{rho-of-N}. The
improvement is that now $P$ comes equipped with the map $\chi (P) :=
\pr_{\G/\TOP} \circ \kappa (P) \co P \ra \G/\TOP$.

Recall the recipe for constructing $\bar N$ from $\chi (N) \co N \ra
\G/\TOP$ repeated in the proof of Proposition \ref{lem:phi-alg-sur}
as steps (1) to (6). In that proof a generalization of steps (1) to
(6) was used when one starts with a map from a manifold with
boundary to $\G/\TOP$.

Using this generalised procedure we construct a manifold with
boundary $\bar P$ with a homotopy equivalence $\varphi (P) \co \bar
P \ra P \times \FF P^2$. The boundary is $\del \bar P =
\sqcup_{i=1}^r \bar N$ since the map $\chi (P) \co P \ra \G/\TOP$,
restricts to $\sqcup_{i=1}^r \chi (N) \co \sqcup_{i=1}^r N \ra
\G/\TOP$ on $\del P = \sqcup_{i=1}^r N$. Furthermore if $\bar p (P)
: \bar P \to P$ denotes the analogue of $\bar p : \bar N \to N$
obtain from the generalised procedure, then we have the map
\[ \lambda (\bar P) = (\lambda (P) \circ \bar p (P)) \co \bar P \ra P \ra BG,\]
which restricts to $\sqcup_{i=1}^r \lambda (\bar N)$ on the
boundary. If follows that $\bar P$ is the desired coboundary of
$\sqcup_{i=1}^r \bar N$ over $\lambda (\bar N)$ which may be used in
\ref{rho-is-mult}.
\end{proof}

\small
\bibliography{rho-add}
\bibliographystyle{alpha}

\end{document}